\documentclass[a4paper,dvipsnames]{article}

\usepackage{amsthm,amssymb,amsmath,enumerate,graphicx,epsf,bm,caption,framed}
\usepackage{authblk}
\usepackage[greek,english]{babel}  
\usepackage[percent]{overpic}
\usepackage{epstopdf}

\newcommand{\COLORON}{0}
\newcommand{\NOTESON}{1}
\newcommand{\Debug}{0}
\usepackage[usenames]{color} 
\usepackage{amsthm,amssymb,amsmath,bbm,enumerate,graphicx,epsf,stmaryrd}
\usepackage[bookmarks, colorlinks=false, breaklinks=true]{hyperref} 

\newcommand{\comment}[1]{}
\newcommand{\COMMENT}[1]{}

\definecolor{darkgray}{rgb}{0.3,0.3,0.3}
\newcommand{\defi}[1]{{\color{darkgray}\emph{#1}}}

\comment{
	\begin{lemma}\label{}	
\end{lemma}
\begin{proof}

\end{proof}

\begin{theorem}\label{}
\end{theorem} 
\begin{proof} 	

\end{proof}

}

\newtheorem{proposition}{Proposition}[section]

\newtheorem{theorem}[proposition]{Theorem}
\newtheorem{corollary}[proposition]{Corollary}

\newtheorem{lemma}[proposition]{Lemma}
\newtheorem{observation}[proposition]{Observation}

\newtheorem{examp}[proposition]{Example}

\newcommand{\FIG}{0}

\ifnum \NOTESON = 1 \newcommand{\note}[1]{ 

\hspace*{-30pt}
	{\color{blue}  NOTE: \color{Turquoise}{\small  \tt \begin{minipage}[c]{1.1\textwidth}  #1 \end{minipage} \ignorespacesafterend }} 
	
	}
\else \newcommand{\note}[1]{} \fi

\newcommand{\afsubm}[1]{ \ifnum \Debug = 1 {\mymargin{#1}}
\fi} 

\ifnum \Debug = 1 
\else  \fi

\ifnum \FIG = 1 \newcommand{\fig}[1]{Figure ``{#1}''}
\else \newcommand{\fig}[1]{Figure~\ref{#1}} \fi

\ifnum \FIG = 1 
\else  \fi

\ifnum \Debug = 1 \usepackage[notref,notcite]{showkeys}
\fi

\ifnum \COLORON = 0 \renewcommand{\color}[1]{}
\fi

\newcommand{\N}{\ensuremath{\mathbb N}}
\newcommand{\R}{\ensuremath{\mathbb R}}
\newcommand{\C}{\ensuremath{\mathbb C}}
\newcommand{\Z}{\ensuremath{\mathbb Z}}

\newcommand{\OO}{\ensuremath{\Omega}}

\makeatletter
\DeclareRobustCommand{\cev}[1]{%
  \mathpalette\do@cev{#1}%
}
\newcommand{\do@cev}[2]{%
  \fix@cev{#1}{+}%
  \reflectbox{$\m@th#1\vec{\reflectbox{$\fix@cev{#1}{-}\m@th#1#2\fix@cev{#1}{+}$}}$}%
  \fix@cev{#1}{-}%
}
\newcommand{\fix@cev}[2]{%
  \ifx#1\displaystyle
    \mkern#23mu
  \else
    \ifx#1\textstyle
      \mkern#23mu
    \else
      \ifx#1\scriptstyle
        \mkern#22mu
      \else
        \mkern#22mu
      \fi
    \fi
  \fi
}

\makeatother

\newcommand{\seq}[1]{\ensuremath{(#1_n)_{n\in\N}}} 

\renewcommand{\Pr}{\mathbb{P}}

\newcommand{\Lr}[1]{Lemma~\ref{#1}}
\newcommand{\Lrs}[1]{Lemmas~\ref{#1}}
\newcommand{\Tr}[1]{Theorem~\ref{#1}}

\newcommand{\Sr}[1]{Section~\ref{#1}}

\newcommand{\Cr}[1]{Corollary~\ref{#1}}

\newcommand{\st}{such that}

\newcommand{\labtequ}[2]{
 \begin{equation} \label{#1} 	\begin{minipage}[c]{0.9\textwidth}  #2 \end{minipage} \ignorespacesafterend \end{equation} }

\newcommand{\mymargin}[1]{
 \ifnum \Debug = 1
  \marginpar{%
    \begin{minipage}{\marginparwidth}\small%
      \begin{flushleft}%
        {\color{blue}#1}%
      \end{flushleft}%
   \end{minipage}%
  }%
 \fi
}%

\newcommand{\mySection}[2]{}

\usepackage{xcolor}

\usepackage{mathtools}

\DeclarePairedDelimiter\abs{\lvert}{\rvert}

\RequirePackage{latexsym} \RequirePackage{amsthm}
\RequirePackage{amsmath}
\usepackage{hyperref}

\usepackage{enumerate}
\RequirePackage{amssymb}\RequirePackage{makeidx}
\usepackage{dsfont}
\usepackage{mathrsfs}

\newcommand{\myremark}[1]{\ifnum \Debug = 1 \tiny #1 \fi}

\newcommand{\ar}[1]{\overrightarrow{#1}}

\newcommand{\ra}[1]{\cev{#1}}

\newcommand{\lat}{\ensuremath{\mathbb{L}}}

\usepackage{authblk}

\title{Convergence of square tilings to the Riemann map}

\author[1]{Agelos Georgakopoulos}
\author[2]{Christoforos Panagiotis}
\affil[1,2]{{Mathematics Institute}\\
        {University of Warwick}\\
        {CV4 7AL, UK}\thanks{Supported by the European Research Council (ERC) under the European Union's Horizon 2020 research and innovation programme (grant agreement No 639046).}\\}
        
\begin{document}
\date{}
\maketitle

\begin{abstract}
A well-known theorem of Rodin \& Sullivan, previously conjectured by Thurston, states that 
the circle packing of the intersection of a lattice with a simply connected planar domain $\OO$ into the unit disc $\mathbb{D}$ converges to a Riemann map from $\OO$  to  $\mathbb{D}$ when the mesh size converges to 0. 
We prove the analogous statement when circle packings are replaced by the square tilings of Brooks et al.
\end{abstract}


\section{Introduction}
In 1987 Thurston \cite{ThuFin} proposed the following method for approximating the Riemann map from a simply connected domain $\OO \subset \C$ to the unit disc $\mathbb{D}$. Let $k \cdot \mathbb{T}$ denote the triangular lattice re-scaled by a factor of $k>0$, and consider the plane graph $G_n:=  \Omega  \cap 2^{-n} \cdot \mathbb{T}$. Consider the sequence of maps $f_n: \OO \to \mathbb{D}$ obtained by circle packing $G_n$ to $\mathbb{D}$ so that its vertices near a fixed point of $\OO$ are mapped to circles near the origin in $\mathbb{D}$, and interpolating from the vertices of $\mathbb{T}$ to all points in $\OO$. 
Thurston \cite{ThuFin}  conjectured that $f_n$ converges to a Riemann map from $\OO$ to $\mathbb{D}$, and this was proved by Rodin \& Sullivan \cite{RodSu}. The aim of this paper is to prove the analogous statement when circle packings are replaced by another discrete version of the Riemann mapping theorem, the square tilings of Brooks et al. \cite{SqTil}.

The theorem of Rodin  \& Sullivan has been extended in various directions. Convergence for lattices other than the triangular was proved by He \& Rodin \cite{HeRodin}, under the assumption of bounded degree. Stephenson \cite{Steph} proved that the convergence of $f_n$ to the Riemann map is locally uniform.  Doyle, He \& Rodin \cite{DHeRodin} improved  the quality of convergence  to convergence in $C^2$. He  \& Schramm \cite{HeSchRie} gave an alternative proof of the convergence in $C^2$, and their proof works in further generality. In particular, it does not need the assumption of bounded degree of \cite{HeRodin}. Finally, for the triangular lattice, He \& Schramm \cite{HeSchInf} proved $C^\infty$-convergence of circle packings to the Riemann map. Our result for square tilings also gives $C^\infty$-convergence. We work with the square lattice $\lat$ for convenience, but our proof applies to any lattice admitting a vertex-transitive action of the group $\Z^2$.

\begin{theorem}\label{informal}
Consider a Jordan domain $\Omega$ in $\C$, and let $s_n: \Omega \to \C$ be defined by linear interpolation of the Brooks et al.\ square tiling map of $\Omega \cap 2^{-n} \cdot \lat$. Then $(s_n)$ converges in $C^\infty(\Omega)$ to a conformal map.
\end{theorem}

Detailed definitions are given in the following sections, before the precise statement of \Tr{informal} is given in \Sr{Proof}.

Our result rests on a discrete version of the following remark. A classical result of Kakutani states that Brownian motion is conformally invariant \cite{MorPer} (up to a time reparametrization that is irrelevant for our purposes). It is known that this conformal invariance is still true when the Brownian motion is reflected from $\partial \Omega$ back into $\Omega$  under the assumption that $\partial \OO$ is in $C^{1,\alpha}$ \cite{RBMConf}. This allows one to describe a Riemann map from a Jordan domain $\Omega$ in $\C$ to a rectangle $H:= [0,\ell]\times [0,1]$ as follows. Consider four distinct points $x_1$, $x_2$, $x_3$ and $x_4$ in $\partial \Omega$ in clockwise ordering. These points subdivide $\partial \Omega$ into four subarcs  $T$, $R$, $B$ and $L$, appearing in that order along $\partial \Omega$. Let $f: \Omega \to H$ be the Riemann map mapping the $x_i$ to the corners of $H$ (where we tacitly use Caratheodory's theorem, see \Sr{complex}). Let $\ell$ denote the extremal length (see \Sr{complex}) in $\OO$ between $L$ and $R$. Then we have
\begin{observation} \label{obs}
For every $z\in \OO$ we have $f(z)= p_{LR}+ ip_{TB}$, where $p_{TB}$ is the probability that a Brownian motion started at $z$ and reflected along $\partial \OO$ will reach $T$ before $B$, and similarly, $p_{LR}$ equals the probability that a Brownian motion started at $z$ and reflected along $\partial \OO$ will reach $L$ before $R$ multiplied by $\ell$. 
\end{observation}
This follows immediately from the conformal invariance of reflected Brownian motion and the fact that the formula is correct when $\OO$ is replaced by the rectangle $H$. Observation~\ref{obs} can be extended to all Jordan domains \OO\ by approximation by a sequence of Jordan subdomains with $C^{1,\alpha}$ boundaries. This follows from the weak convergence of the reflected Brownian motion on the subdomains to the reflected Brownian motion on \OO\ \cite{RBMWeak}, and the converge of the corresponding conformal maps defined on the subdomains to $f$. 

The construction of the square tilings of Brooks et al.\ can be thought of as a discrete variant of Observation~\ref{obs} (and in fact our results can be used to obtain an alternative proof thereof), with reflected Brownian motion replaced by random walk on a mesh $G_n$: it assigns coordinates to vertices and edges of $G$ similarly to the above function $f$. We use a compactness argument to obtain a convergent subsequence, and then verify that any limiting function $f$ satisfies the Cauchy-Riemann equations \eqref{CR} by noticing that $s_n$ satisfies a discrete variant thereof. We then proceed to show that $f$ is a bijection and determine its boundary behaviour by harnessing the combinatorial structure of our model as well as using probabilistic and complex analytic arguments. This uniquely defines $f$ and implies the convergence of the whole sequence $(s_n)$.

The convergence of discrete functions like the ones we use to functions defined in the continuum is by no means a new idea. In 1928 Courant, Friedrichs and Lewy \cite{Courant} considered functions defined in discrete domains as the solutions of some discrete boundary value problems and proved convergence to their continuous counterparts. Since then several authors have considered similar approximation schemes, see e.g.\ \cite{ChelSmi,LF} and the references therein.

\medskip
Apparently, part of the  motivation for Thurston's question leading to the Rodin--Sullivan theorem came from approximating Riemann maps by computer, and he suggested an algorithm for doing so \cite[Appendix 2]{RodSu}. However, circle packing a given graph into a disc is a computationally challenging problem, and according to \cite{ColSteCir},
{\it ``In the numerical conformal mapping of plane regions, it is unlikely that circle packing can ever compete in speed or accuracy with classical numerical methods...''}. On the contrary, computing the square tiling boils down to solving a linear system of equations of size proportional to the number of vertices of the approximating graph $G_n$. We are not yet sure to what extent our algorithm can compete with or complement existing  numerical methods, but we did implement it on a computer and \fig{juliaRM} shows an example of a resulting  approximation of a Riemann map, while \fig{juliaST200} shows the corresponding square tiling.

\begin{figure}
\centering
\includegraphics[width=0.7\textwidth]{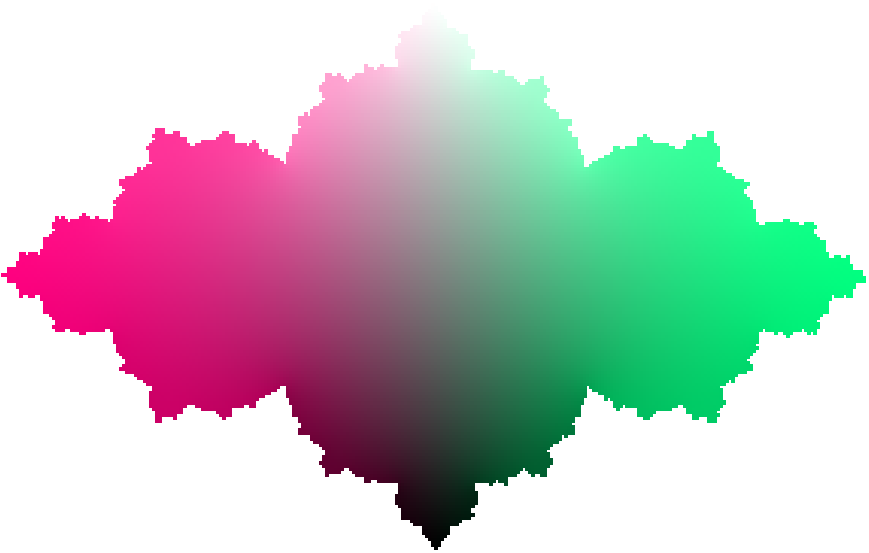}\hspace{.07\textwidth}
\includegraphics[width=0.2\textwidth]{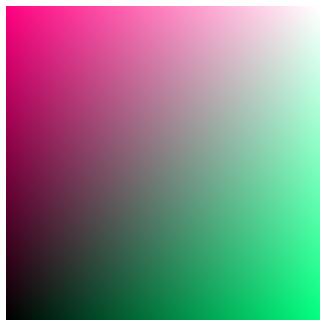}
\caption{An approximation of a Riemann map between a Julia set and a square obtained by implementing our algorithm on Mathematica. Each point in one figure is the image, under the Riemann map, of the unique point in the other figure with the same colour.} \label{juliaRM}
\end{figure}

   \begin{figure}[htbp]
   \centering
   \noindent
   \includegraphics[width=0.65\textwidth, angle=90,origin=c]{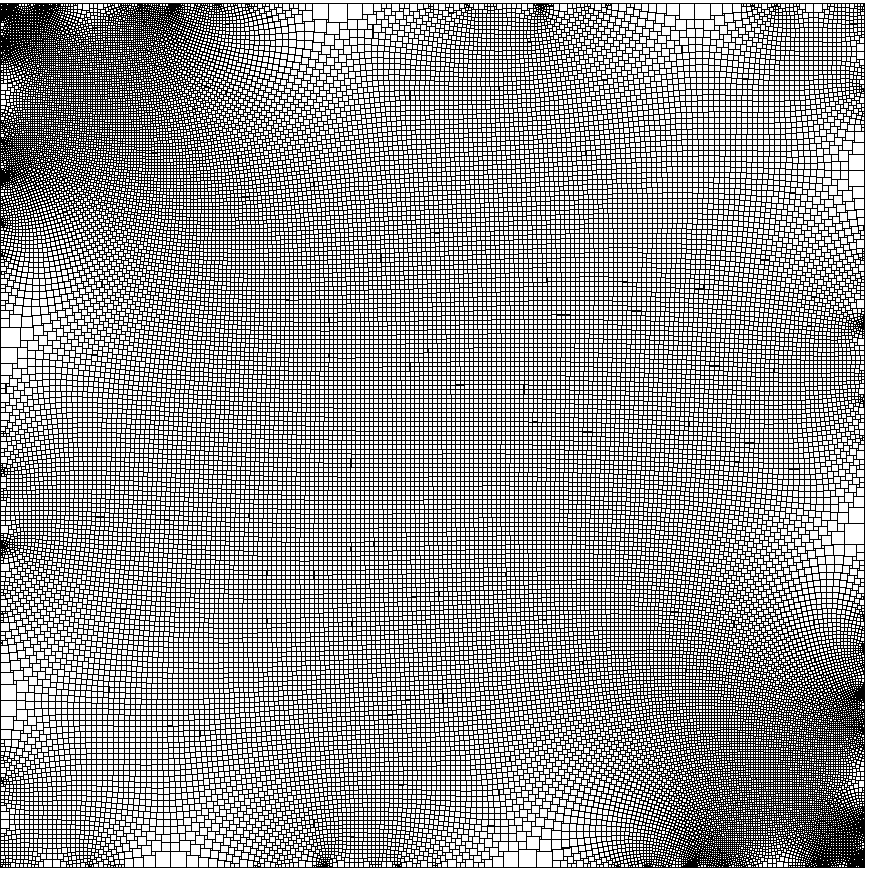}
   \caption{\small The square tiling of a mesh lying inside the Julia set of \fig{juliaRM}. The dark regions consist of large amounts of squares, corresponding to edges along which very little current flows.}
   \label{juliaST200}
   \end{figure}

\medskip
According to \cite{CaFlPaSqu},
{\it ``Riemann, in formulating his famous Riemann mapping theorem, surely relied on the physics of electrical networks and conducting metal plates for motivation.''}
Some biographical evidence about Riemann support this claim. He had a strong interest in the physics of electricity:  {\it ``To complete his Habilitation Riemann had to give a lecture. He prepared three lectures, two on electricity and one on geometry.\footnote{http://www-groups.dcs.st-and.ac.uk/history/Biographies/Riemann.html}''}. Both Riemann and Kirchhoff moved to Berlin in 1847\footnote{http://www-history.mcs.st-andrews.ac.uk/Biographies/Kirchhoff.html}, at a time when the latter was working on his laws of electricity (which we use in \Sr{sec til}). Some of the ideas involved in the construction of square tilings and in our proof support the belief that the physics of electrical networks influenced Riemann in formulating his mapping theorem in his thesis in 1851. Indeed, the quantity $p_{TB}=p_{TB}(z)$ in Observation~\ref{obs} coincides with the voltage $v(z)$ at $z$ when a unit potential difference is imposed between $T$ and $B$, because both functions are harmonic and satisfy the same boundary conditions. The set of points $z$ with $p_{LR}(z)=x\in (0,\ell)$ form a field line of the resulting electrical current.
\section{Preliminaries}

\subsection{Graph theoretic definitions}

Let $G=(V,E)$ be a graph fixed throughout this section, where $V=V(G)$ is its set of vertices and $E= E(G)$ its set of edges. We assume that $G$ is a \defi{plane graph} endowed with a fixed embedding in the plane $\R^2$; more formally, $G$ is a plane graph if $V(G)\subset \R^2$ and each edge $e\in E(G)$ is an arc between its two vertices that does not cross any other edge. 

It is a standard fact that one can associate with the graph $G$ a further plane graph $G^*=(V^*,E^*)$ called the (geometric) \defi{dual} of $G$, defined as follows. We place a vertex of $V^*$ in each face of $G$, and we connect two vertices of $V^*$ with an edge whenever the corresponding faces of $G$ share an edge. Thus there is a bijection $e\mapsto e^*$ from $E$ to $E^*$.

For convenience, we will be working with the square lattice \lat. Its vertex set is the set of points of $\R^2$ with integer coordinates, and its edge-set comprises the horizontal and vertical length $1$ straight line segments connecting them. For an integer $n\geq 0$, we let $2^{-n} \cdot \lat$ denote the plane graph obtained from \lat\ by multiplying the coordinates of each point by $2^{-n}$. Thus each edge of $2^{-n} \cdot \lat$ has length $2^{-n}$. With a slight abuse, we denote both this graph and its vertex set by $2^{-n} \cdot \lat$ for convenience. Notice that \lat\ is a \defi{self-dual} graph, i.e. its dual graph $\lat^*$ is isomorphic to \lat.

\subsection{Complex analytic definitions}\label{complex}

Consider a \defi{simply connected domain} $\Omega \subsetneq \C$, i.e. a connected open set such that its complement $\C\setminus \Omega$ is also connected. The Riemann mapping theorem states that there is a conformal map $\phi$ from $\Omega$ to the unit disk $D$, i.e. a holomorphic and injective function mapping $\Omega$ onto $D$. We will be working with bounded simply connected domains $\Omega$ whose boundary is a simple closed curve $\gamma$ (a homeomorphic image of the unit circle). In this case, $\gamma$ is called a \defi{Jordan curve}, and $\Omega$ is called a \defi{Jordan domain}. A homeomorphic image of the unit interval is called a \defi{Jordan arc}.

Caratheodory studied the boundary behaviour of conformal maps, and established that $\phi$ witnesses the topological properties of the boundary of $\Omega$. In particular, when $\Omega$ is a Jordan domain, Caratheodory's theorem (see e.g. \cite{Krantz,Pom}) states that
$\phi$ extends to a homeomorphism between the closures $\overline{\Omega}$ and $\overline{D}$.

It follows from the Riemann mapping theorem that for every $M>0$, there is a conformal map from $\Omega$ to the rectangle $(0,M)\times (0,1)$. By Caratheodory's theorem, when $\Omega$ is a Jordan domain, $\phi$ extends to a homeomorphism between the closures $\overline{\Omega}$ and $[0,M]\times [0,1]$. Consider now four distinct points $x_1,x_2,x_3,x_4 \in \partial \Omega$ in clockwise ordering, and let $y_1,y_2,y_3,y_4$ be the four corners of $[0,M]\times [0,1]$ in clockwise ordering starting from the top left one. It is natural to ask whether there is a conformal map from $\Omega$ to $(0,M)\times (0,1)$, with $\phi(x_i)=y_i$, $i=1,2,3,4$. As it turns out, three boundary points determine uniquely a conformal map \cite[Corollary 2.7]{Pom}, hence for each choice $x_1,x_2,x_3,x_4$ of boundary points, there is only one value of $M$ (depending on these points) for which a conformal map with the desired property exists. 

To determine the value of $M$, we recall the classical notion of \defi{extremal length}. Let $\overline{x_ix_j}$ denote the arc of $\partial \Omega$ from $x_i$ to $x_j$ traversed in the clockwise direction. To define the extremal length between $\overline{x_1x_2}$ and $\overline{x_3x_4}$, given a Borel-measurable function $\rho:\Omega\to\C$ and a rectifiable curve $\gamma$ in $\Omega$ connecting $\overline{x_1x_2}$ to $\overline{x_3x_4}$, we let
$$L_{\rho}(\gamma):=\int_{\gamma} \rho |dz|,$$
where $|dz|$ denotes the Euclidean element of length. We also define
$$A(\rho):=\iint_{\Omega} \rho^2 dxdy.$$
The extremal length between $\overline{x_1x_2}$ and $\overline{x_3x_4}$ is
$$\sup_{\rho} \inf_{\gamma} \dfrac{L_{\rho}(\gamma)^2}{A(\rho)},$$
where the infimum ranges over all rectifiable curves $\gamma$ in $\Omega$ connecting $\overline{x_1x_2}$ to $\overline{x_3x_4}$, and the supremum ranges over all Borel-measurable functions $\rho:\Omega\to\C$ with $0<A(\rho)<\infty$.

The extremal length between the sides $[0,M]\times \{1\}$ and $[0,M]\times \{0\}$ of the rectangle can be computed explicitly and is equal to $1/M$ 
\cite[p. 52-53]{AhlConformal}. Moreover, the extremal length is conformally invariant \cite[p. 52]{AhlConformal}. Therefore, $M$ is the reciprocal of the extremal length between $\overline{x_1x_2}$ and $\overline{x_3x_4}$.

\subsection{Simple random walk and electrical networks}

A \defi{walk} on $G$ is a (possibly infinite) sequence \seq{v}\ of elements of $V$ \st\ $v_i$ is always connected to $v_{i+1}$ by an edge. The \defi{simple random walk} on $G$ begins at some vertex and when at vertex $x$, traverses one of the edges $\ar{xy}$ incident to $x$ according to the probability distribution
$$p_{x\to y}:= \frac{1}{d(x)},$$
where $d(x)$ denotes the \defi{degree} of $x$, that is, the number of vertices connected to $x$ by an edge. 

There is a well-known correspondence between electrical networks and simple random walk. Given two vertices $p$ and $q$ of $G$, we connect a battery across the two vertices so that the voltage at $p$ is equal to $0$ and the voltage at $q$ is equal to $1$. Then certain currents will flow along the directed edges of $G$ and establish certain voltages at the vertices of $G$. It is a standard fact that for every vertex $u$, the voltage at $u$ is equal to the probability that the simple random walk from $u$ visits $q$ before $p$.

The physical notion of the electrical current can be defined in purely mathematical terms as follows. Let us first denote by $\ar{E}$ the set of ordered pairs $(x,y)$ with $xy\in E$. We write $\ar{xy}$ to denote $(x,y)$. We say that a function $f:\ar{E}\rightarrow \mathbb{R}$ is \defi{antisymmetric}, and write $f:\ar{E}\hookrightarrow \mathbb{R}$, if $f(\ar{xy})=-f(\ra{xy})$ for every $xy\in E$. 
Given two vertices $p$ and $q$ of $G$, we say that a function $f:\ar{E}\hookrightarrow \mathbb{R}$ is a \defi{p-q flow} if it satisfies \defi{Kirchhoff's cycle law}, which postulates that for every vertex $x$ other than $p$ and $q$,
$$\sum_{y\in N(x)} f(\ar{xy})=0,$$
where $N(x)$ denotes the set of neighbours of $x$, i.e. the vertices connected to $x$ by an edge. The \defi{p-q current} is the (unique) p-q flow $i:\ar{E}\hookrightarrow \mathbb{R}$ that satisfies \defi{Kirchhoff's cycle law}.
Kirchhoff's cycle law postulates that for every cycle $C=x_0e_{01}x_1e_{12}x_2\ldots x_n$ in $G$, where the $x_j$ are vertices, the $e_{jk}$ are edges, and $x_n=x_0$, we have
$$\sum_{j=0}^{n-1} i(\ar{x_jx_{j+1}})=0.$$ 

The \defi{intensity} $I^*$ of $i$ is the sum 
$$\sum_{y\in N(x)} i(\ar{xy}).$$ 
The \defi{effective resistance} $R^{\text{eff}}$ between $x$ and $y$ admits several equivalent definitions, among which the most useful for us is
\begin{equation}\label{eff-int}
R^{\text{eff}}=1/I^*.
\end{equation}
Duffin \cite{DuffinEx} proved that the effective resistance coincides with the notion of `discrete extremal length'. We will utilise this fact later on.

\subsection{Discrete partial derivatives and convergence in $C^{\infty}$}

Consider an integer $n\geq 0$. Any function $g$ defined on a subset of $2^{-n}\cdot \lat$ can be extended to the whole of $2^{-n} \cdot \lat$ by setting $g(z)=0$ on the remaining vertices $z$ of $2^{-n} \cdot \lat$. We will always assume that our functions are extended in this way to the whole of $2^{-n} \cdot \lat$. For every vertex $z$ of $2^{-n} \cdot \lat$ we define the functions $\dfrac{\partial g}{\partial x}(z):=2^n \big(g(u)-g(z)\big)$ and $\dfrac{\partial g}{\partial y}(z):=2^n \big(g(v)-g(z)\big)$, where $u=z+2^{-n}$ and $v=z+2^{-n}i$. The functions $\dfrac{\partial g}{\partial x}$ and $\dfrac{\partial g}{\partial y}$ are called the \defi{partial derivatives} of $g$ with respect to $x$ and $y$, respectively. For functions defined on the dual graph $\big(2^{-n} \cdot \lat\big)^*$, the partial derivatives are defined analogously. As usually, by repeatedly applying the operators $\dfrac{\partial}{\partial x}$ and $\dfrac{\partial}{\partial y}$ in any order and any number $k$ of times, we define the partial derivatives of order $k$. 

Consider now a domain $\Omega\subset \C$. We say that a sequence $(f_n)$ of functions defined on $2^{-n} \cdot \lat$ converges in $C^{\infty}(\Omega)$ to a smooth function $f:\Omega\to \C$ if for every closed disk $\Delta\subset\Omega$ and for every $k\geq 0$, the partial derivatives of $f_n$ of order $k$ converge uniformly in $\Delta$ to the corresponding partial derivatives of $f$ of order $k$. For a sequence of functions defined on the dual graph $\big(2^{-n} \cdot \lat\big)^*$, the definition is analogous.

To prove convergence of a sequence of discrete functions we will first extend all functions of the sequence to the whole of $\Omega$, and then apply the Arzel{\`a}-Ascoli theorem \cite{BabyRudin}, which gives necessary and sufficient conditions to decide whether a subsequence of functions converges uniformly.

Let $(f_n)$ be a sequence of continuous functions on a compact set $K\subset \C$. The sequence is said to be equicontinuous if, for every $\epsilon>0$ and $x\in K$, there exists $\delta>0$ such that
$$|f_n(x)-f_n(y)|<\epsilon$$
whenever $|x-y|<\delta$ for every $n\geq 0$.

\begin{theorem}[Arzel{\`a}-Ascoli theorem] 
Let $(f_n)$ be a uniformly bounded and equicontinuous sequence of continuous functions on a compact set $K\subset \C$. Then there is a subsequence of $(f_n)$ that converges uniformly on $K$.
\end{theorem}

\section{Construction of the tiling}\label{constr}

In this section we will recall the construction of the square tiling of Brooks et. al. \cite{SqTil}, and we will formally define the sequence $(s_n)$ featuring in \Tr{informal}. 

Consider a Jordan domain $\Omega$ in $\C$, and four distinct points $x_1$, $x_2$, $x_3$ and $x_4$ in $\partial \Omega$ in clockwise ordering. These points subdivide $\partial \Omega$ into four subarcs $T=\overline{x_1x_2}$, $R=\overline{x_2x_3}$, $B=\overline{x_3x_4}$ and $L=\overline{x_4x_1}$, where $\overline{x_ix_j}$ denotes the arc of $\partial \Omega$ from $x_i$ to $x_j$ traversed in the clockwise direction. We will refer to $T$, $R$, $B$ and $L$ as the \em{top, right, bottom} and \em{left arc} of $\partial \Omega$, respectively. 

Without loss of generality we can assume that the origin lies in $\Omega$. For every $n\geq 0$, we consider the subgraph of $2^{-n} \cdot \lat$ determined by those vertices and edges lying entirely in $\Omega$, and we define $\Omega_n$ to be the connected component of the origin in that graph. The \defi{boundary} $\partial \Omega_n$ of $\Omega_n$ is the set of vertices $z$ of $\Omega_n$ that are incident to an edge in $2^{-n} \cdot \lat$ that intersects $\partial \Omega$. By rescaling $\Omega$ if necessary we can assume that for every $n\geq 0$, no pair of adjacent edges of $2^{-n} \cdot \lat$ intersects opposite arcs of $\partial \Omega$ ($T$ and $B$ or $R$ and $L$). In particular, no edge of $2^{-n} \cdot \lat$ intersects opposite arcs of $\partial \Omega$. We can now define $T_n$, $B_n$ to be the sets of vertices of $\partial \Omega_n$ that are incident to an edge intersecting $T$, $B$, respectively. We also define $R_n$, $L_n$ to be the sets of vertices of $\partial \Omega_n \setminus (T_n\cup B_n)$ that are incident to an edge intersecting $R$, $L$, respectively. It will be useful for the construction of our tiling to identify the elements of $T_n$ and $B_n$ into single vertices, which we will denote by $t_n$ and $b_n$, respectively. We will write $G_n=(V_n,E_n)$ for the graph obtained from $\Omega_n$ after these identifications. The \defi{boundary} $\partial G_n$ of $G_n$ is the set of boundary vertices obtained after these identifications, i.e. $\partial G_n:=\{t_n,b_n\}\cup R_n \cup L_n$.

\subsection{The dual graph $G_n^*$} \label{secDual}

We consider $G_n$ as a plane graph, in other words, $V_n$ is now a set of points of $\R^2$ and $E_n$ is a set of arcs in $\R^2$ each joining two points in $V_n$. The points in $\R^2$ occupied by the elements of $V_n$ and $E_n$ can be chosen in such a way that the vertices $t_n$ and $b_n$ are incident with the unbounded face of $G_n$, the position of every other vertex of $V_n$ remains the same after the identifications, and the arcs connecting vertices of $V_n\setminus \{t_n,b_n\}$ are straight lines. It will be useful for the construction of the square tiling to associate to $G_n$ a new graph $G_n^*$ by slightly modifying the standard definition of the dual graph of $G_n$. First, let $t_nb_n$ be an arc in $\R^2$ connecting $t_n$ with $b_n$, every interior point of which lies in the unbounded face of $G_n$. Consider the graph $G_n'=(V_n,E_n \cup\{t_nb_n\})$, and let $(V^*_n,E^*_n \cup\{l_nr_n\})$ be the dual graph of $G_n'$, where $l_n$ is the face incident to $L_n$, and $r_n$ is the face incident to $R_n$. Deleting the edge $l_nr_n$ we obtain the graph $G^*_n=(V^*_n,E^*_n)$. 
 
Notice that there is a bijection $e\mapsto e^*$ from $E_n $ to $E^*_n$. The orientability of the plane allows us to extend the bijection $e\mapsto e^*$
between $E_n$ and $E^*_n$ to a bijection between the directed edges of $G_n$ and $G^*_n$ in such a way that if $\ar{E}(u)$ is the set of edges incident to a vertex $u$ of $G_n$ directed towards $u$, then $\{\ar{e}^* \mid \ar{e}\in \ar{E}(u) \}$ is a cycle oriented in the counter-clockwise direction.

\subsection{The tiling} \label{sec til}

For any vertex $u\in V(G_n)$, let $h_n(u)$ be the probability that simple random walk in $G_n$ starting from $u$ hits $t_n$ before $b_n$. Thus $h_n(t_n)=1$ and $h_n(b_n)=0$. Notice that $h_n$ is \defi{harmonic} at every vertex in $V(G_n)\setminus \{t_n,b_n\}$, i.e.
$$h_n(u)=\frac{1}{d(u)}\sum_{v\in N(u)} h_n(v).$$ The values of $h_n$ will be used as `height' coordinates in the construction of the square tiling. Before defining the `width' coordinates, let us consider the Ohm dual of $h_n$, namely the flow $w_n$ given by the relation
$$w_n(\ar{xy})=h_n(x)-h_n(y)$$ for every directed edge $\ar{xy}$ of $G_n$. Observe that $w_n$ is antisymmetric, i.e. $w_n(\ar{xy})=-w_n(\ar{yx})$, and satisfies Kirchhoff's laws. 

Let us now define the functions $w'_n$ and $h'_n$, the values of which will be used as `width' coordinates. Given a directed edge $\ar{xy}$ in the dual graph $G_n^*$, we let $$w'_n(\ar{xy})=w_n(\ar{xy}^*),$$ where $\ar{xy}^*$ is the directed edge of $G_n$ corresponding to $\ar{xy}$. It is a well-known known consequence of the duality between Kirchhoff's laws in the primal and the dual graph that the function $w'_n$ satisfies both Kirchhoff's cycle law and Kirchoff's node law. To define $h'_n$, set first $h'_n(l_n)=0$. For every other vertex $z\in V(G_n^*)$, pick a path $P_z= z_0 z_1 \ldots z_k$, where $z_0=z$ and $z_k=l_n$, and let $h'_n(z)= \sum_{i<k} w'_n(\ar{z_i z_{i+1}})$. The value of $h'_n(z)$ does not depend on the choice of the path $P_z$, because $w'_n$ satisfies Kirchhoff's cycle law.

It is not hard to see that the pair $h'_n$, $w'_n$ satisfies Ohm's law, i.e. $w'_n(\ar{xy})=h'_n(x)-h'_n(y)$. Since $w'_n$ satisfies Kirchhoff's node law, we deduce that $h'_n$ is a harmonic function on the set $V(G_n^*)\setminus \{l_n,r_n\}$. Furthermore, it follows from the definition of $h'_n$ that
$$h'_n(r_n)=I^*_n:=\sum_{z\in N(t_n)} w_n(\ar{t_n z}),$$
since the directed edges $\ar{t_n z}^*$ form a directed path from $r_n$ to $l_n$. 

Having defined $w'_n$ and $h'_n$, we can now specify the squares $S_e$ of our square tiling indexed by the edges of $G_n$. Consider an edge $e=xy\in E(G_n)$ and assume that $h_n(x)\geq h_n(y)$. Then the square $S_e$ has the form $I_e\times [h_n(x),h_n(y)]$. To define $I_e$, we consider the dual edge $e^*=x'y'$ of $e$, and we let $I_e$ be the interval $[h'_n(x'),h'_n(y')]$, noting that $h'_n(y')\geq h'_n(x')$. For every $u\in V(G_n)$, we define $$I_u=\cup_{e\in E(u)} I_e,$$ where $E(u)$ is the set of edges incident to $u$. It is easy to check that $I_u$ is an interval. Brooks et. al. \cite{SqTil} proved that the interiors of the squares are disjoint and the union of the squares is the rectangle $[0,I^*_n]\times [0,1]$. In other words, the collection $S=\{S_e, e\in E(G_n)\}$ is a tiling of the rectangle $[0,I^*_n]\times [0,1]$. See also \cite{BeSchrHar,planarPB} for a similar construction of a square tiling of a cylinder.

We remark that $h'_n(z)$ coincides with $I^*_n p_n(z)$ for every vertex $z$ of $G_n^*$, where $p_n(z)$ denotes the probability that simple random walk from $z$ visits $r_n$ before $l_n$. This follows from observing that both functions are harmonic at every vertex $z\neq l_n,r_n$, and coincide at $l_n$ and $r_n$, because $p_n(l_n)=0$ and $p_n(r_n)=1$. This easily implies  

\begin{lemma} \label{rotate}
The square tiling of $G_n^*$ with respect to $l_n,r_n$ coincides with that of $G_n$ rotated by 90 degrees  and re-scaled by $1/I^*_n$.
\end{lemma}
\subsection{Definition of the interpolation $s_n$} \label{def int}

We will now define the functions $s_n:\Omega\to \C$. Given a vertex $z$ of $\Omega_n$ that does not lie in $T_n\cup B_n$, we define the imaginary part of $s_n(z)$ to be equal to $h_n(z)$. For the real part of $s_n(z)$, let $f_1,f_2,\ldots f_k$ be the faces incident to $z$ in $G_n$, where $k$ denotes the number of such faces. The real part of $s_n(z)$ is defined to be the average horizontal coordinate $\sum_{i=1}^k h'_n(z_i)/k$. For those vertices $z$ lying in $T_n\cup B_n$, we define $s_n(z)$ in terms of $h_n$ and $h_n'$ in a similar manner, except that we now replace $z$ by $t_n$ or $b_n$, as appropriate.

Notice that when $I_z$ is not a single point, $s_n(z)$ belongs to the interior of $I_z \times \{h_n(z)\}$. To extend $s_n$ to all of $\Omega$, we first set $s_n$ to be equal to $0$ on the remaining vertices of $2^{-n} \cdot \lat$, and then extend it to every point in $\Omega$ (in fact to every point in $\C$) by linear interpolation. Thus, if $(x_1,y_1)$, $(x_2,y_1)$, $(x_2,y_2)$ and $(x_2,y_1)$ are the four corners of a square in $2^{-n} \cdot \lat$ in counter-clockwise ordering with $(x_1,y_1)$ being the bottom left one, and $(x,y)$ is a point lying in that square, then
\begin{equation}\label{inter}
\begin{gathered}
s_n(x,y)=4^n\Big((x_2-x)(y_2-y)s_n(x_1,y_1)+(x-x_1)(y_2-y)s_n(x_2,y_1)+\\
(x-x_1)(y-y_1)s_n(x_2,y_2)+(x_2-x)(y-y_1)s_n(x_1,y_2)\Big).
\end{gathered}
\end{equation}
We remark that every point $z\in \Omega$ is mapped under $s_n$ in the rectangle $[0,I^*_n]\times [0,1]$, since this holds for the lattice points $2^{-n} \cdot \lat$ and the rectangle $[0,I^*_n]\times [0,1]$ is a convex set.

We end this section by pointing out that the square tiling $S$ gives rise to a plane graph living in $[0,I^*_n]\times [0,1]$ and being isomorphic to a certain subgraph of $G_n$. To be more precise, consider those vertices $u$ of $G_n$ for which $I_u$ is not a single point, and place a vertex at $s_n(u)$. It follows from the definition of $s_n$ that $s_n(u)$ lies in the interior of the segment $I_u\times \{h_n(u)\}$ for every such $u$. It is not hard to see that for every edge $e=uv$ of $G_n$ with $S_e$ not being a single point, an arc can been drawn connecting $s_n(u)$ to $s_n(v)$, all points of which other than its endpoints do not belong to some segment $I_u\times \{h_n(u)\}$, and no pair of such arcs intersects (except at a vertex). This follows easily from the fact that the interiors of the squares in $S$ are disjoint; see \cite{SqTil} for a detailed proof. On the other hand, if $S_e$ degenerates to a single point, then the segments $I_u$ and $I_v$ share a common endpoint but are otherwise disjoint. Hence we can connect $s_n(u)$ to $s_n(v)$ with a straight horizontal line. In this way, we obtain a plane graph that is isomorphic to the subgraph of $G_n$ spanned by those vertices $u$ of $G_n$ for which $I_u$ is not a single point.

\section{Proof of main result}\label{Proof}

In this section we  prove \Tr{informal}. The proof is split into several smaller parts. Let us start by formulating it more precisely.

\begin{theorem}\label{main}
Consider a Jordan domain $\Omega$ in $\C$, and four distinct points $x_1$, $x_2$, $x_3$ and $x_4$ in $\partial \Omega$ in clockwise ordering. Let $E$ be the extremal length between the arcs $\overline{x_1x_2}$ and $\overline{x_3x_4}$ of $\partial \Omega$, and $y_1$, $y_2$, $y_3$, $y_4$ be the four corners of the rectangle $[0,1/E]\times [0,1]$ in clockwise ordering with $y_1$ being the top left one. Then the sequence $(s_n)$ converges in $C^\infty(\Omega)$ to the conformal map $f$ mapping $\Omega$ onto the rectangle $(0,1/E)\times (0,1)$, with $f(x_i)=y_i$, $i=1,2,3,4$.
\end{theorem}

\subsection{Convergence to a holomorphic map}

Since $s_n$ is defined via $h'_n$ and $h_n$, it will be useful to first establish the convergence of $h'_n$ and $h_n$. The following lemma is our first step in that direction.

\begin{lemma}\label{conv}
There is a strictly increasing sequence $(k_n)$ of natural numbers such that
both $h'_{k_n}$ and $h_{k_n}$ converge in $C^{\infty}(\Omega)$ to smooth functions $u:\Omega\to \mathbb{R}$ and $v:\Omega\to \mathbb{R}$, respectively, and the intensities $I_{k_n}$ converge to a non-negative real number $I$.
\end{lemma}

For every $n\geq 0$, let $D_n$ be the subgraph $\lat$ spanned by the vertices $(x,y)$ with both $|x|\leq 2^n$ and $|y|\leq 2^n$. In order to prove \Lr{conv} above, we will utilise the next result about the partial derivatives of harmonic functions on $D_n$.

\begin{theorem}\cite{Brandt,LawEst}\label{derivative}
There is a constant $C>0$ such that for every harmonic function $f$ on $D_n$ we have
$$\Bigl\lvert\dfrac{\partial f}{\partial x}(0)\Bigr\rvert \leq \dfrac{C {\lVert f \rVert}_{\infty}}{2^n} \quad \text{ and } \quad \Bigl\lvert\dfrac{\partial f}{\partial y}(0)\Bigr\rvert \leq \dfrac{C {\lVert f \rVert}_{\infty}}{2^n}.$$
\end{theorem}

Since by definition ${\lVert h_n \rVert}_{\infty}\leq 1$, \Tr{derivative} gives after a suitable rescaling and translation of $D_n$ that for every closed disk $\Delta\subset \Omega$, the partial derivatives of $h_n$ on $\Delta$ are bounded. It follows from the next lemma that $h'_n$ is uniformly bounded too.

\begin{lemma}\label{bounded}
There is a constant $c=c(\Omega)>0$ such that $I^*_n \leq c$ for every $n\geq 0$.
\end{lemma}
\begin{proof}
Since $I^*_n$ is by definition equal to the reciprocal of the effective resistance $R^{\text{eff}}_n$ between $t_n$ and $b_n$ \eqref{eff-int}, it suffices to bound $R^{\text{eff}}_n$ from below by a strictly positive real number. 

Duffin \cite{DuffinEx} proved that the effective resistance coincides with the \defi{discrete extremal length}; applied to $G_n$, this statement becomes
$$R^{\text{eff}}_n=\max_{W} \min_{P} \dfrac{\Big( \sum_{e\in E(P)} W_e \Big)^2}{\sum_{e\in E(G_n)} W_e^2},$$
where the minimum ranges over all paths connecting $t_n$ with $b_n$, and the maximum ranges over all assignments $W_e\in [0,\infty)$, $e\in E(G_n)$\footnote{The physical intuition behind this is the classical formula $R^{\text{eff}}_n=\dfrac{V^2}{E}$, where $V$ is the potential difference of an electrical current and $E$ its energy, combined with the fact that the electrical current is the energy minimiser among all functions $W_e$ achieving a given potential difference between the source and the sink.}. For every $n\geq 0$ we will assign to each edge of $G_n$ some positive parameter $W_e=W_e(n)$, and we will show that $$\dfrac{\Big( \sum_{e\in E(P)} W_e \Big)^2}{\sum_{e\in E(G_n)} W_e^2}$$ remains bounded from below for every path $P$ connecting $t_n$ with $b_n$. To this end, for every $e\in E(G_n)$, let $W_e=2^{-n}$. Notice that every path $P$ connecting $t_n$ with $b_n$ gives rise to a path in $\Omega_n$ connecting $T_n$ to $B_n$ which we will still denote by $P$. The sum 
$\sum_{e\in E(P)} W_e$ is now equal to the length of $P$. Hence this sum is bounded from below by the Hausdorff distance between $T_n$ and $B_n$. This distance converges to the distance between the arcs $T$ and $B$, which is strictly positive, showing that for every $n\geq 0$, $\Big( \sum_{e\in E(P)} W_e \Big)^2$ remains bounded from below by a strictly positive constant. 

It remains to bound the denominator $\sum_{e\in E(G_n)} W_e^2$ from above. Associate to each edge $e\in E(G_n)$ the square of side length $2^{-n}$ that contains $e$ (when viewed as an edge in $\Omega_n$) and is dissected by $e$ into two congruent rectangles. The area of each of these squares is equal to $W_e^2=4^{-n}$. It is not hard to see that the interiors of any pair of squares associated to distinct parallel edges of $G_n$ (horizontal or vertical) are disjoint. Moreover, all squares constructed in this way have distance at most $1$ from a point in $\Omega$, hence they lie in a bounded region $\Omega'$ independent of $n$. Therefore,
$$\sum_{e\in E(G_n)} W_e^2 \leq 2 \text{area}(\Omega'),$$
where the factor $2$ comes from considering the squares associated to horizontal edges and those associated to vertical edges.
This completes the proof.
\end{proof}

Since $h'_n=I^*_n p_n$ and ${\lVert p_n \rVert}_{\infty}\leq 1$, \Lr{bounded} and \Tr{derivative} imply that for every closed disk $\Delta\subset \Omega$, the partial derivatives of $h'_n$ on $\Delta$ are uniformly bounded. We are now ready to prove \Lr{conv}.

\begin{proof}[Proof of \Lr{conv}]
The sequence $(I^*_n)$ is bounded by \Lr{bounded}, hence there is a subsequence $(I^*_{k_n})$ of $(I^*_n)$ and a non-negative real number $I$ such that $I_{k_n}$ converges to $I$. Moreover, both $h_{k_n}$ and $p_{k_n}$ are positive harmonic functions bounded from above by $1$. We will use \Tr{derivative} to prove that both $(h_{k_n})$ and $(p_{k_n})$ have further subsequences converging in $C^{\infty}(\Omega)$.

Extend the functions $h_{k_n}$ and $p_{k_n}$ in $\Omega$ by linear interpolation as in \eqref{inter}. Denote these extensions by $H_{k_n}$ and $P_{k_n}$, respectively. Notice that for any point $r$ lying in the interior of some horizontal edge $zy$ of $G_n$, where $y=z+2^{-n}$, we have 
$$\dfrac{\partial H_{k_n}}{\partial x}(r)=\dfrac{\partial h_{k_n}}{\partial x}(z)$$
and
$$\dfrac{\partial P_{k_n}}{\partial x}(r)=\dfrac{\partial p_{k_n}}{\partial x}(z),$$
where in the left-hand side we have the standard partial derivative and in the right-hand side we have the discrete one.
Similar equalities hold for the partial derivatives with respect to $y$.
Moreover, for every point $r$ in the interior of some square of $2^{-n} \cdot \lat$, the partial derivatives of $H_{k_n}$ and $P_{k_n}$ at $r$ are linear combinations of their partial derivatives at the boundary of the square. \Tr{derivative} now implies that the partial derivatives of $H_{k_n}$ and $I^*_{k_n}P_{k_n}$ are locally bounded. Thus the sequences $(H_{k_n})$ and $(I^*_{k_n}P_{k_n})$ are locally bounded and equicontinuous. The Arzel{\`a}-Ascoli theorem now says that the sequences $(H_{k_n})$ and $(I^*_{k_n}P_{k_n})$, hence $(h_{k_n})$ and $(I^*_{k_n}p_{k_n})$, have further subsequences converging locally uniformly to some continuous functions $u:\Omega\to \mathbb{R}$ and $v:\Omega\to\mathbb{R}$, respectively. For convenience, we will assume without loss of generality that the sequences converge along $(k_n)$.

To deduce the convergence in $C^{\infty}(\Omega)$, we observe that if $h$ is a harmonic function defined in a ball around a vertex of $2^{-n} \cdot \lat$, then the function $g(z)=h(z+2^{-n})$ is harmonic at every vertex of the ball except possibly for those at its boundary. It follows that the partial derivatives of $h$ are also harmonic at every vertex of the ball except possibly for those at its boundary, as differences of harmonic functions. This implies that for every $k\geq 0$, all partial derivatives of $h_{k_n}$ and $p_{k_n}$ of order $k$, are harmonic functions on suitable subsets of $V(G_n)$. It is easy to prove inductively using \Tr{derivative} that all partial derivatives of order $k$ of $h_{k_n}$ and $I^*_{k_n}p_{k_n}$ are locally bounded. Arguing as above we deduce that all partial derivatives of order $k$ of $h_{k_n}$ and $I^*_{k_n}p_{k_n}$ have a further subsequence that converges locally uniformly. It follows by \Lr{calculus} below that the limiting functions are the corresponding partial derivatives of order $k$ of $u$ and $v$, respectively. In other words, all subsequential limiting functions coincide with the corresponding partial derivatives of order $k$ of $u$ and $v$, respectively. This implies that all partial derivatives of order $k$ of $h_{k_n}$ and $I^*_{k_n}p_{k_n}$ converge locally uniformly to the corresponding partial derivatives of order $k$ of $u$ and $v$, respectively. This completes the proof.
\end{proof}

We now state the lemma mentioned in the proof of \Lr{conv} above, which is an easy exercise.

\begin{lemma}\label{calculus}
Consider a sequence of piecewise continuously differentiable functions $f_n:[a,b]\to \R$. Assume that there are continuous functions $f,g:[a,b]\to \R$ such that $(f_n)$ converges uniformly to $f$, and $(f'_n)$ converges uniformly to $g$. Then $f$ is differentiable with $f'=g$.
\end{lemma}

We fix a sequence $(k_n)$, smooth functions $u:\Omega\to \R$ and $v:\Omega\to \R$, and a constant $I$ as in \Lr{conv}, and we let $f:\Omega\to\C$ be the function defined as $f=u+iv$. In the next lemma we show that $(s_{k_n})$ converges to $f$.

\begin{lemma}\label{s conv}
The sequence $(s_{k_n})$ converges in $C^{\infty}(\Omega)$ to $f$.
\end{lemma}
\begin{proof}
Recall that for every vertex $z\in V(G_n)$, the real part of $s_n$ is equal to $\sum_{i=1}^k h'_n(f_i)/k$. Every partial derivative of order $j$ of $\sum_{i=1}^k h'_n(f_i)/k$ at $z$ is a linear combination of the corresponding partial derivatives of order $j$ of $h'_n$
at $f_i$. Hence every partial derivative of order $j$ of the real part of $s_{k_n}$ converges locally uniformly to the corresponding partial derivative of order $j$ of $u$ by \Lr{conv}. The imaginary part of $s_n$ is by definition equal to $h_n$, hence converges in $C^{\infty}(\Omega)$ to $v$. Thus we obtain the desired result.
\end{proof}

Our aim is to show that $f$ is the conformal map of \Tr{main}. We first show that $f$ is holomorphic.

\begin{lemma}\label{holomorphic}
The function $f:\Omega\to \C$ is holomorphic.
\end{lemma}
\begin{proof}
It follows from \Lr{conv} that $f\in C^1(\Omega)$ (in fact $f\in C^{\infty}(\Omega)$). We will verify that $f$ satisfies the Cauchy-Riemann equations
\begin{align} \label{CR}
\dfrac{\partial u}{\partial x}=\dfrac{\partial v}{\partial y} \quad \text{ and } \quad \dfrac{\partial u}{\partial y}=-\dfrac{\partial v}{\partial x},
\end{align}
which implies that $f$ is holomorphic.

Consider a point $z\in \Omega$ with \defi{dyadic} coordinates $z=(k2^{-m},l2^{-m})$, $k,l,m\in \mathbb{N}$. Notice that for every $n$ large enough, $z$ is occupied by a vertex of $G_n$ of degree $4$. Let $\ar{E}(z)$ be the four edges of $G_n$ incident to $z$ directed towards $z$. Let $C=C(n)$ be the dual directed cycle oriented in the counter-clockwise direction. Recall that for $\ar{rz}\in\ar{E}(z)$ we have
$$h'_n(r')-h'_n(z')=h_n(r)-h_n(z),$$ where $r'$ and $z'$ are the endvertices of the dual directed edge $\ar{rz}^*\in E(C)$ of $\ar{rz}$. We can use this property to deduce that the pair $h'_n$, $h_n$ satisfies the following `discrete Cauchy-Riemann' equations
\begin{align} \label{CRd}
\dfrac{\partial h'_n}{\partial x}(z+(-1+i)2^{-n-1})=\dfrac{\partial h_n}{\partial y}(z) \quad \text{ and } \quad \dfrac{\partial h'_n}{\partial y}(z+(1-i)2^{-n-1})=-\dfrac{\partial h_n}{\partial x}(z).
\end{align}
Taking limits along the sequence $(k_n)$ of \Lr{conv} we deduce that $f$ satisfies the Cauchy-Riemann equations at $z$. The continuity of the partial derivatives of $f$ combined with the density of the set of points with dyadic coordinates implies that $f$ satisfies the Cauchy-Riemann equations at every point of $\Omega$.
\end{proof}

In view of \eqref{CRd}, we can deduce that the partial derivatives of $h'_n$ are locally bounded just from the fact that the partial derivatives of $h_n$ are locally bounded, namely without using \Lr{bounded}. However, \Lr{bounded} is necessary to ensure that the subsequential limits of the intensities $I^*_n$ are indeed bounded.

\subsection{Boundary behaviour}

Consider a point $z\in \partial \Omega$. We say that $y\in\C$ is an $f$-limit point of $z$ if there is a sequence $(z_k)$ in $\Omega$ converging to $z$ such that $f(z_k)$ converges to $y$. Write $T'$ (resp. $B'$, $L'$, $R'$) for the top (resp. bottom, left, right) side of the rectangle $[0,I]\times [0,1]$.

Using the weak convergence of simple random walk to Brownian motion, we prove the following lemma about the $f$-limit points of $\partial \Omega \setminus \{x_1,x_2,x_3,x_4\}$. To be more precise, consider simple random walk on the graph $2^{-n} \cdot \lat$, and let $S_n(t)$, $t=0,1,\ldots$ denote its position at time $t$. Extend $S_n(t)$ on the whole interval $[0,\infty)$ by linear interpolation, and define the process $W_n(t)=S_n(4^n t)$. It is an easy application of Donsker's invariance principle \cite{MorPer} that $W_n$ converges weakly in the locally uniform topology to the $2$-dimensional Brownian motion $W$. Notice that since we are rescaling our lattice, we do not need to further scale $S_n$ to obtain the convergence.

\begin{lemma}\label{limit points}
Consider a point $z\in \partial \Omega \setminus \{x_1,x_2,x_3,x_4\}$. If $z\in U$, where $U\in\{T,B,L,R\}$, then all $f$-limit points of $z$ lie in $U'$.
\end{lemma}
\begin{proof}
Consider a point $y\in \Omega$. We claim that 
\begin{align}\label{ineq}
\begin{split}
v(y)\geq \Pr_y(\tau_T=\tau_{\partial \Omega}), \quad v(y)\leq 1-\Pr_y(\tau_B=\tau_{\partial \Omega}), \\
u(y)\geq I\Pr_y(\tau_R=\tau_{\partial \Omega}) \quad \text{ and } \quad u(y)\leq I\big(1-\Pr_y(\tau_L=\tau_{\partial \Omega})\big),
\end{split}
\end{align}
where $\Pr_y$ denotes the probability measure of Brownian motion starting from $y$, and $\tau_S$ denotes the first hitting time of a set $S$.
We will prove only the first inequality. The remaining ones follow similarly.

Assume first that $y$ has dyadic coordinates, and let $n$ be large enough that $y$ is occupied by a vertex of $G_n$. Clearly, $h_n(y)$ is at least the probability for simple random in $G_n$ to hit $t_n$ before hitting $\partial G_n\setminus \{t_n\}$. Notice that simple random walk in $G_n$ up to the fist hitting time of $\partial G_n$ behaves like simple random in $2^{-n} \cdot \lat$ up to the first hitting time of $\partial \Omega_n$. Hence $h_n(y)$ is at least the probability $\Pr_{n,y}(\tau_{T_n}=\tau_{\partial \Omega_n})$, where $\Pr_{n,y}$ denotes the probability measure of simple random walk in $2^{-n} \cdot \lat$ starting from $w$.

We will now prove that 
\labtequ{prob}{$\Pr_{n,y}(\tau_{T_n}=\tau_{\partial G_n})$ converges to $\Pr_y(\tau_T=\tau_{\partial \Omega})$,} using the weak convergence in the locally uniform topology of $W_n$ to Brownian motion. Indeed, there is a coupling of the simple random walk and Brownian motion in the same probability space, such that $W_n$ converges almost surely to Brownian motion in the locally uniform topology by virtue of Skorokhod's representation theorem \cite{Bill}. Notice that a priori it is possible for $W_n$ to exit $\Omega_n$ at $T_n$ for every $n$ large enough, even though $W$ exits $\Omega$ at $\partial \Omega\setminus T$. Our aim is to show that this is almost surely never the case. 

To this end, let $U\in \{T,B,L,R\}$ denote the boundary arc first visited by Brownian motion, which is almost surely well-defined since $W$ exits $\Omega$ at $\{x_1, x_2, x_3, x_4\}$ with probability $0$. By the almost sure continuity of the Brownian paths, there is a number $\delta>0$ such that the Hausdorff distance between the compact sets $\{W(t), t\in [0,\tau_{\partial \Omega}+\delta]\}$ and $(\partial \Omega\setminus U)\cup \{x_1, x_2, x_3, x_4\}$ is strictly positive. Hence for every $n$ large enough, the distance between the sets $\{S_n(t), t\in [0,4^n (\tau_{\partial \Omega}+\delta)]\}$ and $\partial \Omega_n \setminus U_n$ is strictly positive. For every point $p\in\partial \Omega$, Brownian motion from $p$ exits $\overline{\Omega}$ immediately, i.e. 
$$\Pr_p(\inf \{t>0 \mid W(t)\in \mathbb{C}\setminus \overline{\Omega}\}=0)=1$$
by \Lr{regular} below.
Therefore, there is some $t$ between $\tau_{\partial \Omega}$ and $\tau_{\partial \Omega}+\delta$ such that $W(t)$ lies in the complement of $\overline{\Omega}$ by the strong Markov property. We can now deduce that for every $n$ large enough, $S_n(t)$ lies in the complement of $\overline{\Omega}$ for some $t$ between $4^n\tau_{\partial \Omega}$ and $4^n(\tau_{\partial \Omega}+\delta)$, hence hits $\partial \Omega_n$ before time $4^n(\tau_{\partial \Omega}+\delta)$. Consequently, $\tau_{U_n}=\tau_{\partial \Omega_n}$ for $S_n$ when $n$ is large enough, because $S_n$ does not hit $\partial \Omega_n \setminus U_n$ by time $4^n(\tau_{\partial \Omega}+\delta)$. This implies that the indicator of the event that $W_n$ exits $\partial \Omega_n$ at $U_n$ converges almost surely to the indicator of the event that $W$ exits $\partial \Omega$ at $U$. Taking expectations we obtain \eqref{prob}. 

Thus when $y$ has dyadic coordinates, the desired inequality\\ $v(y)\geq \Pr_y(\tau_T=\tau_{\partial \Omega})$ follows from the convergence of $h_{k_n}$ to $v$. The continuity of both $v$ and $\Pr_y(\tau_T=\tau_{\partial \Omega})$ gives the inequality for all $y$ in $\Omega$.

To obtain the assertion of the lemma, it remains to show the convergence of $\Pr_y(\tau_U=\tau_{\partial \Omega})$ to $1$ as $y$ tends to $z$. Consider the conformal map from $\Omega$ to $(0,M)\times (0,1)$ of \Tr{main}, where $M=1/E$, and notice that it maps any arc $U\in \{T,B,L,R\}$ to $U'$. Since Brownian motion is conformally invariant \cite{MorPer}, it suffices to prove the assertion when $\Omega$ is the rectangle $(0,M)\times (0,1)$, $U$ is its top side, and $z$ is some interior point of the top side. Write $y=y_1+iy_2$, and let $W=W_1+iW_2$ be our $2$-dimensional Brownian motion, where $W_1$ and $W_2$ are independent $1$-dimensional Brownian motions. Let also $\Pr_{W_i,y_i}$, $i=1,2$ denote the probability measure of $W_i$ starting from $y_i$. Notice that if $W_2$ hits $1$ before $0$, and $W_1$ hits $0$ or $M$ after $W_2$ hits $1$, then $W$ exits the rectangle from the top. The probability
$\Pr_{W_2,y_2}(\tau_1<\tau_0)$ of the first event is equal to $y_2$ (see e.g. \cite[Theorem 2.45]{MorPer}),
which converges to $1$ as $y$ tends to $z$.
Moreover, for every $r>0$ we have $$\Pr_{W_2,y_2}(\tau_1\leq r)=\Pr_{W_2, y_2}(|W_2(r)|\geq 1-y_2)=2\Phi\Big(\frac{1-y_2}{\sqrt{r}}\Big)$$
$$\Pr_{W_1,y_1}(\tau_0\leq r)=\Pr_{W_1,y_1}(|W_1(r)|\geq y_1)=2\Phi\Big(\frac{y_1}{\sqrt{r}}\Big)$$
$$\Pr_{W_1,y_1}(\tau_M\leq r)=\Pr_{W_1,y_1}(|W_1(r)|\geq M-y_1)=2\Phi\Big(\frac{M-y_1}{\sqrt{r}}\Big)$$ by the Reflection Principle (see \cite[Thoerem 2.18]{MorPer}), where
$\Phi$ is the cumulative distribution of the standard Gaussian random variable.
Choosing $r=1-y_2$ we obtain that $\Pr_{W_2,y_2}(|W_2(r)|\geq 1-y_2)$ converges to $1$ as $y$ converges to $z$. On the other hand, since both $y_1$ and $M-y_1$ remain bounded away from $0$, the probabilities $\Pr_{W_1,y_1}(|W_1(r)|\geq y_1)$ and
$\Pr_{W_1,y_1}(|W_1(r)|\geq M-y_1)$ converge to $0$. By the union bound the probability $\Pr_{W_1,y_1}(\tau\leq r)$ converges to $0$ as well, where $\tau$ is the minimum of $\tau_0$ and $\tau_M$. Thus with probability converging to $1$ as $y$ converges to $z$, $W_2$ hits $1$ before $0$, and $W_1$ hits $0$ or $M$ after $W_2$ hits $1$. This proves the desired convergence.
\end{proof}

We now prove the lemma mentioned in the proof of \Lr{limit points} above.

\begin{lemma}\label{regular}
For every $p\in\partial \Omega$ we have 
$\Pr_p(\inf \{t>0 \mid W(t)\in \mathbb{C}\setminus \overline{\Omega}\}=0)=1$.
\end{lemma}
\begin{proof}
Since $\partial \Omega$ is a Jordan curve, every boundary point $p\in\partial \Omega$ is regular for Brownian motion \cite[Problem 2.16]{KaraSh}, i.e.
$$\Pr_p(\inf \{t>0 \mid W(t)\in \mathbb{C}\setminus \Omega\}=0)=1,$$
which is slightly weaker than the desired assertion. To remedy this, let $\gamma$ be a Jordan curve passing through $p$ with the property that every other point of $\gamma$ lies in $\mathbb{C}\setminus \overline{\Omega}$. The existence of such a curve follows easily from the fact that $p$ is accessible by an arc in $\mathbb{C}\setminus \overline{\Omega}$, since $\partial \Omega$ is a Jordan curve. Let $\Omega'$ be the bounded component of $\C\setminus \gamma$. We have that 
$$\Pr_p(\inf \{t>0 \mid W(t)\in \mathbb{C}\setminus \Omega'\}=0)=1$$
as above.
Since every point of $\mathbb{C}\setminus\Omega'$ other than $p$ lies in $\mathbb{C}\setminus \overline{\Omega}$, and Brownian motion almost surely never visits $p$ after time $0$, we obtain the desired result.
\end{proof}

\subsection{Proof of injectivity}

\Lr{limit points} implies that $f$ is not a constant function in $\Omega$, hence we obtain

\begin{corollary}\label{zeros}
The number of zeros of $f'$ in any compact subset $K$ of $\Omega$ is finite.
\end{corollary}
\begin{proof}
Consider a compact set $K\subset \Omega$, and assume that $f'$ has infinitely many zeros in $K$. This implies that the zeros of $f'$ have an accumulation point in $K\subset \Omega$. It is a standard fact that then $f'$ is identically zero in $\Omega$. Hence $f$ is a constant function in $\Omega$, contradicting \Lr{limit points}.
\end{proof}

In the next lemma we prove that $f$ is injective in $\Omega$. We offer two proofs of this statement, the first more analytic and the second more combinatorial

\begin{lemma}\label{injective}
The function $f:\Omega\to \C$ is injective.
\end{lemma}
\begin{proof}[Proof 1]
Consider $z,y\in \Omega$ with $z\neq y$. Let $\Delta$ be a small enough open disk centred at $z$ such that $\overline{\Delta}\subset \Omega$ and $y\not \in \overline{\Delta}$. Since $f'$ has finitely many zeros in $\overline{\Delta}$ by \Cr{zeros}, we can reduce the radius of $\Delta$ if necessary to ensure that $f'$ has no zeros in $\overline{\Delta}\setminus \{z\}$. 

We claim that there is a point $p$ in $\Omega\setminus \overline{\Delta}$ such that $f(p)$ lies in the unbounded region of $\mathbb{C}\setminus f(\overline{\Delta})$. Indeed, $f(\overline{\Delta})$ is a compact set, as the continuous image of a compact set. Moreover, $f(\Omega)$ is an open set by the Open mapping theorem for holomorphic functions \cite{PapaRudin}, and contains $f(\overline{\Delta})$. Therefore, $f(\Omega)$ contains a point in the unbounded region of $\mathbb{C}\setminus f(\overline{\Delta})$, which proves the claim.

We can always assume that $p$ has dyadic coordinates and $f'(p)\neq 0$, since otherwise we can replace $p$ with another nearby point $q$ with dyadic coordinates, such that $q\in\Omega\setminus \overline{\Delta}$, $f'(q)\neq 0$ and $f(q)$ lies in the unbounded component of $\mathbb{C}\setminus f(\overline{\Delta})$. 

Let us now choose a connected open set $S$ containing $\Delta$, $y$ and $p$, such that $\overline{S}\subset \Omega$. It is not hard to see that $S$ can be chosen in such a way that $f'$ has no zeros in $\overline{S}\setminus \{z,y\}$ by virtue of \Cr{zeros}. This implies that $f$ is locally injective in $\overline{S}\setminus \{z,y\}$ \cite[Theorem 10.30]{PapaRudin}. Consider a point $q$ in $\Delta\setminus \{z\}$. For every $x\in S$ with $x\neq q$ we have $f(x)\neq f(q)$ by \Lr{local} below. This implies that for every $x\in S\setminus \Delta$, $f(x)$ does not lie in $f\big(\Delta\setminus \{z\}\big)$. Moreover, $f$ maps open sets to open sets by the Open mapping theorem for holomorphic functions. It follows that $f(y)\neq f(z)$, because otherwise some nearby point of $y$ is mapped under $f$ to the open set $f\big(\Delta\setminus \{z\}\big)$. This proves the desired assertion.
\end{proof}

We now prove the lemma mentioned in the proof of \Lr{injective} above.

\begin{lemma}\label{local}
Consider a point $q$ in $\Delta\setminus \{z\}$. Then for every $x\in S$ with $x\neq q$ we have $f(x)\neq f(q)$.
\end{lemma}
\begin{proof}
Let $Q\subset \Delta\setminus \{z\}$ be a small enough square containing $q$ in its interior, such that $f$ is injective at an open neighbourhood of $Q$, and all corners of $Q$ have dyadic coordinates. Consider small enough open disks $\Delta_1$ and $\Delta_2$ centred at $z$ and $y$, respectively, such that the set $S\setminus \big(\Delta_1 \cup \Delta_2\big)$ is connected and contains $Q$. Since $f'$ has no zeros in the closure of $S\setminus \big(\Delta_1 \cup \Delta_2\big)$ and $|f'|$ is a continuous function, there is a constant $c>0$ such that $|f'|$ is bounded from below by $c$ in the closure of $S\setminus \big(\Delta_1 \cup \Delta_2\big)$. The partial derivatives of $h_{k_n}$ converge uniformly in $S\setminus \big(\Delta_1 \cup \Delta_2\big)$ to the partial derivatives of $v$, and $$|f'|=\sqrt{\Big(\dfrac{\partial v}{\partial x}\Big)^2+\Big(\dfrac{\partial v}{\partial y}\Big)^2}$$ by the Cauchy-Riemann equations. It follows that for every $k_n$ large enough, 
$$\sqrt{\Big(\dfrac{\partial h_{k_n}}{\partial x}\Big)^2+\Big(\dfrac{\partial h_{k_n}}{\partial y}\Big)^2}\geq c/2$$ in $S\setminus \big(\Delta_1 \cup \Delta_2\big)$, implying that for every such $k_n$ and for every vertex $x$ of $G_{k_n}$ lying in $S\setminus \big(\Delta_1 \cup \Delta_2\big)$, some of the two partial derivatives of $h_{k_n}$ at $x$ is not $0$. In particular, the line segment $I_x$ is not a single point.

Assume that $k_n$ is large enough that the above property is satisfied, and let $\Gamma_{k_n}$ be the subgraph of $G_{k_n}$ spanned by those vertices lying in the set $S\setminus \big(\Delta_1 \cup \Delta_2\big)$. As described in \Sr{constr}, we can construct a plane isomorphic copy of $\Gamma_{k_n}$ living in $[0,I^*_{k_n}]\times [0,1]$, which we will denote by $\Gamma'_{k_n}$, in such a way that the restriction of $s_{k_n}$ to the vertex set of $\Gamma_{k_n}$ lifts to an isomorphism between $\Gamma_{k_n}$ and $\Gamma'_{k_n}$. 

Pick some $x\in S$ with dyadic coordinates which does not belong to $Q$, $\Delta_1$ or $\Delta_2$. Let $Q_{k_n}$ be the subgraph of $\Gamma_{k_n}$ spanned by the vertices in $Q$, and let $C_{k_n}$ be the subgraph of $\Gamma_{k_n}$ spanned by the vertices in $\partial Q$. Increasing $k_n$ if necessary we can assume that the following hold:
\begin{enumerate}
\item The corners of $Q$ are occupied by vertices of $G_{k_n}$, hence $C_{k_n}$ spans a cycle.
\item The points $p$ and $x$ are occupied by vertices of $G_{k_n}$.
\item There is a path in $\Gamma_{k_n}$ connecting $p$ and $x$ that does not intersect $Q_{k_n}$.
\item The point $s_{k_n}(p)$ belongs to the unbounded face of the subgraph of $\Gamma'_{k_n}$ spanned by the vertices in $s_{k_n}(Q_{k_n})$.
\end{enumerate} 
The third item follows from the connectedness of $S\setminus Q$, and the last item follows from the fact that $f(p)$ lies in the unbounded component of $\mathbb{C}\setminus f(Q)$ by our choice of $p$. We claim that for those $k_n$ and for every vertex $r$ in the component of $p$ in $G_{k_n}\setminus Q_{k_n}$,
\labtequ{claim}{$s_{k_n}(r)$ lies in the unbounded face of the subgraph of $\Gamma'_{k_n}$ spanned by the vertices in $s_{k_n}(Q_{k_n})$.}
To see this, notice that if $s_{k_n}(r)$ lies in a bounded face of the subgraph spanned by the vertices in $s_{k_n}(Q_{k_n})$, then any path in $\Gamma'_{k_n}$ connecting $s_{k_n}(r)$ to $s_{k_n}(p)$ has to intersect $s_{k_n}(Q_{k_n})$, but there is a path in $\Gamma_{k_n}$ connecting $r$ and $p$ that does not intersect $Q_{k_n}$. This contradicts the fact that $s_{k_n}$ is an isomorphism between 
$\Gamma_{k_n}$ and $\Gamma'_{k_n}$, hence $s_{k_n}(r)$ lies in the unbounded face of the subgraph spanned by the vertices in $s_{k_n}(Q_{k_n})$, as claimed. In particular, $s_{k_n}(x)$ lies in the unbounded face of the subgraph of $\Gamma'_{k_n}$ spanned by the vertices in $s_{k_n}(Q_{k_n})$.

It is not hard to see that 
\labtequ{claim2}{$s_{k_n}(q)$ lies in the face bounded by the cycle spanned by $s_{k_n}(C_{k_n})$.} Assuming not, we can argue as above to deduce that all vertices in the interior of $Q$, which span a connected graph, are mapped under $s_{k_n}$ to the unbounded face of the cycle spanned by $s_{k_n}(C_{k_n})$. This leads to a contradiction as follows. Consider a corner $c$ of $Q$, and notice that the remaining vertices of $C_{k_n}$ together with a vertex in the interior of $Q$ are mapped under $s_{k_n}$ to the vertex set of a cycle that surrounds $s_{k_n}(c)$, and no other vertex of $\Gamma'_{k_n}$. Therefore, $s_{k_n}(c)$ has only two neighbours in $\Gamma'_{k_n}$, namely those in $s_{k_n}(C_{k_n})$, because $\Gamma'_{k_n}$ is a plane graph. However, $c$ has four neighbours in $\Gamma_{k_n}$. This contradiction implies \eqref{claim2}.

Combining \eqref{claim} with \eqref{claim2} we deduce that the distance of $s_{k_n}(x)$ from $s_{k_n}(q)$ is at least the Hausdorff distance between the cycle $s_{k_n}(C_{k_n})$, when viewed as a Jordan curve in the plane, and $s_{k_n}(q)$. Taking limits along $(k_n)$ we obtain that 
\labtequ{claim3}{the distance of $f(x)$ from $f(q)$ is at least the distance of $f(\partial Q)$ from $f(q)$.} Making the radii of the disks $\Delta_1$ and $\Delta_2$ approach $0$, we obtain that \eqref{claim3} holds for all $x\in S\setminus \{z,y,q\}$ with dyadic coordinates, and some appropriate square $Q$ depending on $x$. Combining the continuity of $f$ with the density of the set of points with dyadic coordinates, we obtain \eqref{claim3} for every $x\in S\setminus\{q\}$. The distance of $f(\partial Q)$ from $f(q)$ is strictly positive, as $f$ is injective at an open neighbourhood of $Q$. This implies the desired result.
\end{proof}

We proceed with our second proof of \Lr{injective}, for which we need the following lemma.

\begin{lemma}\label{max side}
Consider a point $z\in \Omega$, and let $\Delta\subset \Omega$ be a closed disk centred at $z$. Let $W(n)=W(n,z,\Delta)$ be the maximum side length of a square $S_e$ in the square tiling of $G_n$ over those edges $e$ with both endvertices in $\Delta$. Then $W(n)$ converges to $0$.
\end{lemma}
\begin{proof}
Notice that the side length of any square $S_e$ in the square tiling of $G_n$ is equal to either $2^{-n}\Bigl\lvert \dfrac{\partial h_n}{\partial x}(p)\Bigr\rvert$ or $2^{-n}\Bigl\lvert \dfrac{\partial h_n}{\partial y}(p)\Bigr\rvert$ for some vertex $p$ of $G_n$, depending on whether $e$ is a vertical or a horizontal edge. We know that there is a constant $C>0$ such that $\Bigl\lvert \dfrac{\partial h_{n}}{\partial x}(p)\Bigr\rvert \leq C$ and $\Bigr\rvert \dfrac{\partial h_{n}}{\partial y}(p)\Bigr\rvert \leq C$ for every $p$ in $\Delta$ by \Tr{derivative}. The desired assertion follows.
\end{proof}

We now give our second proof of \Lr{injective}.
\begin{proof}[Proof 2]
Suppose, to the contrary, there are points $z,y\in \Omega$ with $z\neq y$ and $f(z)=f(y)=h'+ih$ for some $h\in [0,I]$, $h' \in [0,1]$.
\labtequ{neq 01}{We have $h'+ih\in (0,I)\times (0,1)$,}
by \eqref{ineq}, because all of $\Pr_y(\tau_T=\tau_{\partial \Omega})$, $\Pr_y(\tau_B=\tau_{\partial \Omega})$, $\Pr_y(\tau_L=\tau_{\partial \Omega})$ and\\ $\Pr_y(\tau_R=\tau_{\partial \Omega})$ are strictly positive.

Our aim is to find a countable set $X\subset \Omega$ that accumulates to either $z$ or $y$ on which $f$ is constant, as this contradicts the fact that $f$ is a non-constant holomorphic function.

For this, pick a sequence $(z_n)_{n\in\mathbb{N}}$ of vertices of $G_n$ \st\ $z_n$ is incident with a face or edge containing $z$. Pick a sequence $(y_n)_{n\in\mathbb{N}}$ of vertices of $G^*_n$, i.e.\ faces of $G_n$, \st\ $y_n$ contains $y$ in its closure. Let also $D_z,D_y\subset \Omega$ be two closed disks centred, at $z,y$, respectively. Then for every large enough $n$, $z_n$ belongs to $D_z$, and $y_n$ belongs to $D_y$. We will define a sequence of paths $P_n\subset \OO$ in $G_n \cup G^*_n$ along which the values of $f$ are closer and closer to $f(z)=f(y)$, and obtain the desired $X$ as a set of accumulation points of $P_n$.

\Lr{conv} implies that $\lim_n h_{k_n}(z_{k_n}) = h$.
Let $a_n:= Re(s_n(z_n))$ be the first coordinate of $z_n$ in the interpolation $s_n$ of the square tiling map as introduced in \Sr{def int}, and recall that $a_n$ was defined as the average of the $h'_n$ values of the faces incident with $z_n$. Therefore, by \Lr{conv} combined with \Lr{max side} we know that $a_{k_n}$ converges to $h'$. Similarly, we have $\lim_n h'_{k_n}(y_{k_n}) = h$, and recalling that $b_n:= H_{n}(y_n)$ is a convex combination of the $h_{k_n}$ values of the vertices incident with $y_n$, we obtain $\lim_n b_{k_n} = h$ by \Lrs{conv} and \ref{max side}. We will assume without loss of generality that $Im(s_n(z_n))\leq Im(s_n(y_n))$.

\begin{figure}[!ht] 
   \centering
   \noindent

\begin{overpic}[width=.6\linewidth]{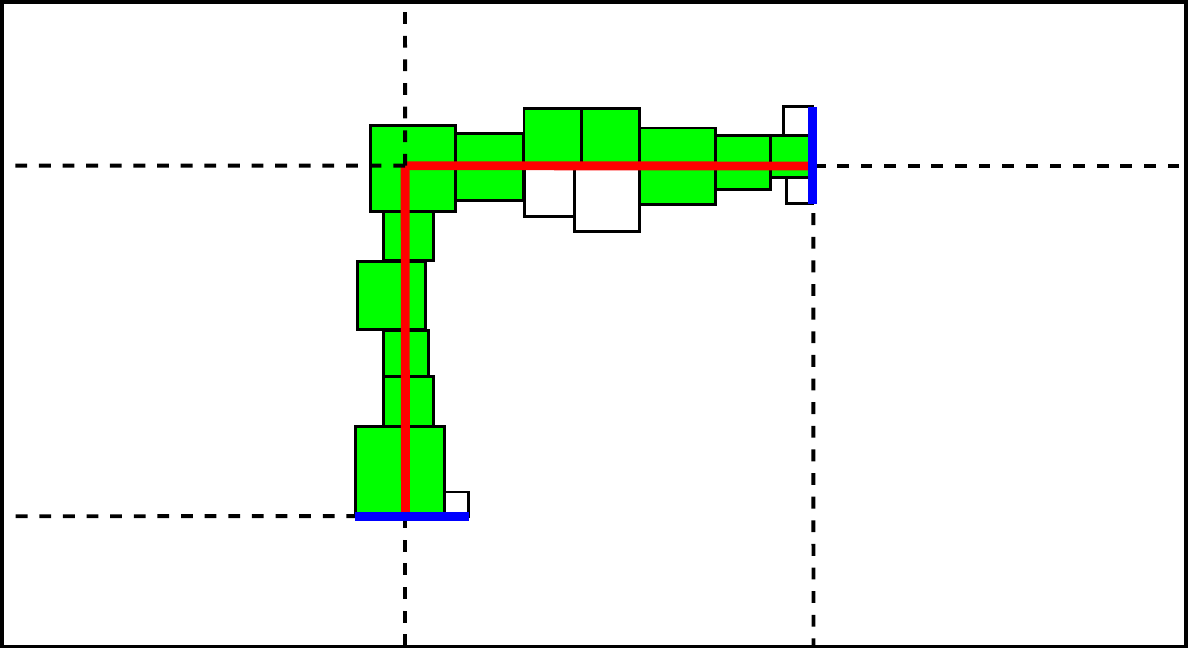} 
\put(-7,40){$b_n$}
\put(-16,10){$H_n(z_n)$}
\put(32.5,-5){$a_n$}
\put(63,-6){$h'_n(y_n)$}
\end{overpic}

\

\

\begin{overpic}[width=.6\linewidth]{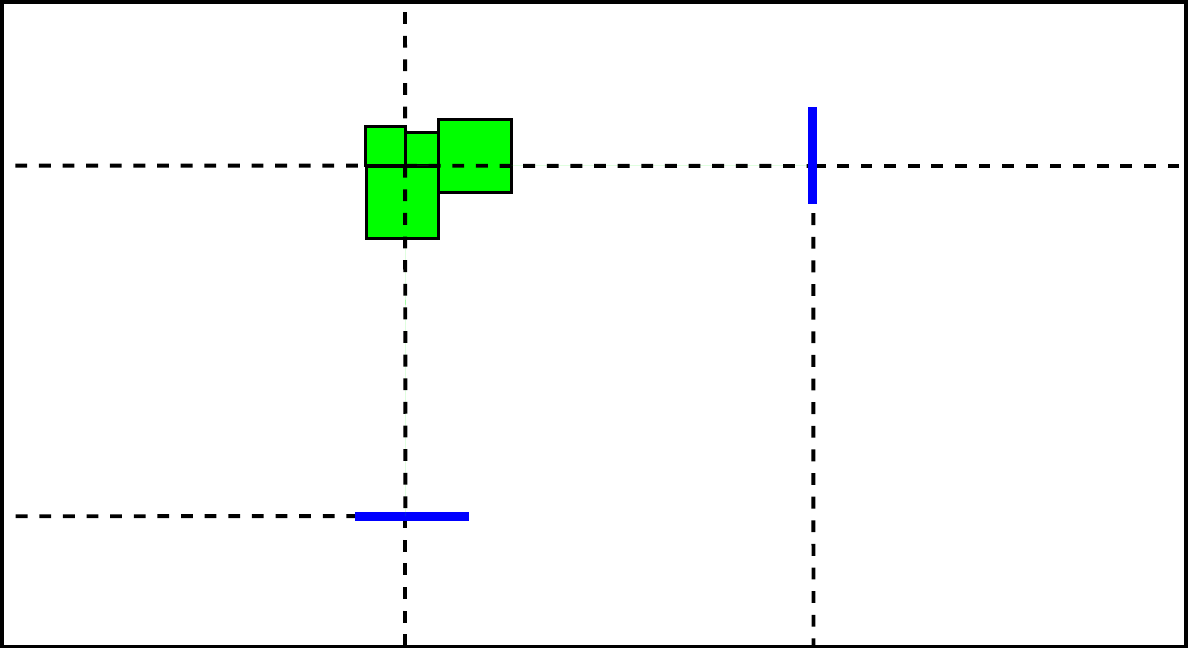} 
\end{overpic}
\caption{The situation in the proof of \Lr{injective}. The shaded squares depict the constructed path $P_n$. The top figure corresponds to the simpler case where $(a_n,b_n)$ lies in the interior of a square of the tiling. The bottom figure shows how $P_n$ is adapted locally in the other case.}\label{path}
\end{figure}

In order to construct $P_n$, we will consider two cases according to the behaviour of $s_n$ on $D_z,D_y$. Consider at first the case that some $I_w$, $w\in \{z_n,y_n\}$ is trivial, i.e. a single point, and furthermore, that there is
\labtequ{P'_n}{a path $P'_n$ in $G_n$ or $G^*_n$ connecting $w$ to the complement of $D_w$, such that $I_u$ is a single point for every vertex $u$ of $P'_n$.} Then we let $P_n=P'_n$. 

It remains to consider the case that no such path $P'_n$ exists. We will obtain $P_n$ by combining two paths $Q_n\subset G_n$ and $Q_n^*\subset G_n^*$. We start the construction of $Q_n$ by the set of edges $e$ whose square $S_e$ in the square tiling of $G_n$ has non-zero width, and contains a point ---possibly at its boundary--- with coordinates $(a_n,\eta)$ with $\eta \in [h_n(z_n), b_n]$. If for some value of $\eta$ there are two such edges $e,f$, which can only be the case when both $S_e,S_f$ are tangent on either side the vertical line $L$ with first coordinate $a_n$, we only keep the edge that lies on the left of $L$. We let $E_n$ denote the set of edges thus obtained. We remark that
\labtequ{up left}{$E_n$ must contain an edge $e$ \st\ $S_e$ contains a point $(\zeta,\eta)$ with $\zeta< a_n$ and $\eta>b_n$, }
because some $S_e$ contains $(a_n, b_n)$ as well as such a point (see \fig{path}).

Next, we claim that $E_n$ spans a path of $G_n$. Indeed, we can linearly order the edges $e\in E_n$ by the second coordinates in $S_e$, and note that any two consecutive edges $e_1,e_2$ in this ordering share a vertex $v$, namely one with $h(v)$ coinciding with the vertical coordinate of $S_{e_1} \cap S_{e_2}$. We let $Q_n'$ denote this path.

We take $Q_n=Q_n'$ if the latter happens to start at $z_n$, but this will fail if the interval $I_{z_n}$ is a single point $I_{z_n}=\{ a_n \}$. In this case, we let $Z_n$ be the connected component of $z_n$ in $G_n$ consisting of vertices $u$ such that $I_u$ is trivial, and let $\partial Z_n$ denote the set of those vertices of $G_n\setminus Z_n$ that have a neighbour in $Z_n$. It is not hard to see that $I_u=I_{z_n}$ for every $u\in V(Z_n)$. Moreover, for every vertex $v$ of $\partial Z_n$ we have that $I_v$ is non-trivial and contains $a_n$, hence there is an edge $e$ incident with $v$ such that the square $S_e$ is non-trivial, and contains $s_n(z_n)$ as well as a point with larger second coordinate than $s_n(z_n)$. Now our assumption that no path $P'_n$ as defined in \eqref{P'_n} exists, implies that $Z_n$ lies in $D_z$. Thus for every large enough $n$, $Z_n$ has at least two boundary vertices, and so there is a boundary vertex $q_n$ and an edge $e$ incident with $q_n$ such that $S_e$ is non-trivial, and contains $s_n(z_n)$ as well as a point $(c,d)$ with $c<a_n$ and $d>Im(s_n(z_n))$. Pick now a path $R_n$ in $Z_n\cup \{q_n\}$ connecting $z_n$ to $q_n$. Then $Q_n:= R_n \cup Q_n'$ is a path starting at $z_n$, and we have completed the first half of the definition of $P_n$.

The other half is now easy: we recall that the square tiling of $G_n^*$ with respect to $l_n,r_n$ coincides with that of $G_n$ rotated by $90$ degrees and re-scaled by $1/I^*_n$ by \Lr{rotate}, and repeat the same construction with the role of $z_n$ and $y_n$ interchanged, to obtain the path $Q_n^*\subset G_n^*$ starting at $y_n$.

We claim that $Q_n, Q_n^*$ intersect when viewed as subsets of \OO. Indeed, they both traverse the unique edge $e$ \st\ $S_e$ contains the point $(a_n,b_n)$ as well as a point $(\zeta,\eta)$ with $\zeta>a_n, \eta>b_n$. 
Therefore, $Q_n \cup Q_n^*$ contains a $z_n$--$y_n$~arc in \OO, which is our $P_n$. Write $L_n$ for the curve in $(0,I)\times (0,1)$ connecting $s_n(z_n)$ to $s_n(y_n)$ that lies in the union of the lines $x=a_n$ and $y=b_n$.

We can assume without loss of generality that in both cases $Q_n$ contains the vertex $z_n$. Consider now a positive integer $m$, and let $\Delta_m$ denote the closed annulus centred at $z$ of radii $1/2m$ and $1/(2m+1)$. Notice that for every $m$ large enough, $\Delta_m$ lies entirely in $D_z$, and $y_n$ lies in the unbounded component of $\C\setminus \Delta_m$. Moreover, for every $n$ large enough, $z_n$ lies in the bounded component of $\C\setminus \Delta_m$. Therefore, for every such $m$ and $n$, there is a point of $Q_n$ contained in $\Delta_m$. Pick such a point and denote it by $x_n(m)$. We can choose $x_n(m)$ to be a vertex of either $G_n$ or $G^*_n$. 

For every fixed $m$, the sequence $(x_n(m))_{n\in\mathbb{N}}$ has an accumulation point in $\Delta_m$, which we denote by $x(m)$. Let $X$ be the set of all $x(m)$. Notice that all $x(m)$ are pairwise distinct, because the annuli $\Delta_m$ are by definition disjoint, and furthermore $x(m)$ converges to $z$ as $m\to \infty$. 

By construction, for every fixed $m$, the values of $s_{k_n}$ at $x_{k_n}(m)$ are close to $f(z)=f(y)$: the points $s_{k_n}(z_{k_n})$ and $s_{k_n}(y_{k_n})$ converge to $f(z)=f(y)$, and the coordinates of every point of $L_{k_n}$ are bounded from above and below by the coordinates of $s_{k_n}(z_{k_n})$ and $s_{k_n}(y_{k_n})$. Therefore, the points of $L_{k_n}$ converge uniformly to $f(z)=f(y)$. Furthermore, the distance between $s_{k_n}(x_{k_n}(m))$ and $L_{k_n}$ converges to $0$ by \Lr{max side}, hence $s_{k_n}(x_{k_n}(m))$ converge to $f(z)=f(y)$. Since $s_{k_n}$ converges uniformly in $\Delta_m$ to $f$, we have that $f(x(m))$ is an accumulation point of the sequence $\big(s_{k_n}(x_{k_n}(m))\big)_{n\in\mathbb{N}}$. Thus, $f(x(m))=f(z)=f(y)$. This proves that $f$ is constant on $X$, as desired.
\end{proof}

\subsection{Behaviour at the designated boundary points}

We will now determine the behaviour of $f$ near $x_1,x_2,x_3$ and $x_4$. The proof of the next lemma is based purely on the boundary behaviour of $f$ at $\partial \Omega \setminus \{x_1,x_2,x_3,x_4\}$, and the fact that $f$ is a conformal map.

\begin{lemma}\label{corners}
For each $i=1,2,3,4$, the only $f$-limit point of $x_i$ is $y_i$.
\end{lemma}
\begin{proof}
Let us assume without loss of generality that $i=1$. Consider a Riemann map $\phi$ from the open unit disk $D$ onto $\Omega$, and recall that $\phi$ extends to a homeomorphism between their closures $\overline{D}$ and $\overline{\Omega}$ by Caratheodory's theorem. Let $X_1=\phi^{-1}(x_1)$, and define $g=f\circ \phi$, which is a conformal map, as $f$ is conformal by \Lr{holomorphic} and \Lr{injective}. For each $r>0$, consider the curve $C(r)=D\cap \{z\in \mathbb{C} \mid \abs{z-X_1}=r\}$, and let $l(r)$ be the length of the curve $g\big(C(r)\big)$. Write $t(r)$ for the set $\{t\in [0,2\pi] \mid X_1+re^{it}\in D\}$ on which the standard parametrization $X_1+re^{it}$ of $C(r)$ is defined. Since $g$ is a conformal map, the integral $$\int_{0}^{1/2} \dfrac{l(r)^2}{r} dr$$ is finite. This follows from applying the Cauchy-Schwarz inequality to the formula $$l(r)=\int_{t(r)} |g'(X_1+re^{it})|rdt,$$ and then using the fact that
$$\iint_{D} |g'(z)|^2 dxdy=\text{area}(g(D)),$$
which is finite; see e.g. \cite[Lemma 5.1.3]{Krantz} for a detailed proof. On the other hand, the function $1/r$ is not integrable at $0$, hence there is a sequence $(r_k)$ of strictly positive real numbers converging to $0$, such that the sequence $\big(l(r_k)\big)$ converges to $0$ as well.

For every $k$ large enough, $l(r_k)$ is in particular finite, implying that $g(C(r_k))$ extends to a continuous curve $\gamma_k$ defined on the closure of $t(r_k)$, which clearly has the same length as $g(C(r_k))$. Furthermore, for every $k$ large enough, one of the two endpoints of $C(r_k)$ lies in $\phi^{-1}(T)$, while the other lies in $\phi^{-1}(L)$. Hence one of the two endpoints of $\gamma_k$ lies in $T'$, while the other lies in $L'$ by \Lr{limit points}. Notice that the endpoints of $\gamma_k$ may possibly coincide, in which case they coincide with $y_1$, but the curve is otherwise injective. Consequently, $\gamma_k$ is either a Jordan arc or a Jordan curve. In both cases, $\gamma_k$ divides $(0,I)\times (0,1)$ into two components $S_1$ and $S_2$. We claim that one of them, say $S_1$, decreases to the empty set as $k\to\infty$. Indeed, if $\gamma_k$ is a Jordan curve, then we let $S_1$ be the component bounded by $\gamma_k$, and otherwise we let $S_1$ be the component whose boundary contains $y_1$. Since the length of $\gamma_k$ converges to $0$, the distance of its endpoints converges to $0$ as well, which is possible only when the endpoints of $\gamma_k$ converge to $y_1$. In both cases, the distance of each of the boundary points of $S_1$ from $y_1$ converges uniformly to $0$. Our claim now follows easily.

Clearly, $C(r_k)$ divides $D$ into two components $C_1$ and $C_2$ as well, with $C_1$ decreasing to the empty set and $C_2$ increasing to $D$ as $k\to\infty$. Since $g$ is injective, one of the sets $g(C_1)$, $g(C_2)$ lies in $S_1$, while the other lies in $S_2$. To decide which one lies in $S_1$, notice that $g(C_2)$ increases to $g(D)$. Therefore, $g(C_2)$ cannot lie in $S_1$ when $k$ is large enough, hence $g(C_1)$ lies in $S_1$ for those $k$. This implies that all possible $g$-limit points of $X_1$, hence all possible $f$-limit points of $x_1$, belong to the closure of $S_1$. The closure of $S_1$ decreases to $\{y_1\}$ as we have seen, and the desired assertion follows.
\end{proof}

\subsection{Completing the proof}

We are now ready to prove \Tr{main}.

\begin{proof}[Proof of \Tr{main}]
It follows from \Lr{holomorphic} and \Lr{injective} that $f$ is a conformal map. Moreover, $f$ maps open sets to open sets by the Open mapping theorem for holomorphic functions. This shows that $f(\Omega)$ is an open set, and furthermore the only boundary points of $f(\Omega)$ are the $f$-limit points of $\partial \Omega$. We can now deduce from \Lr{limit points} and \Lr{corners} that the boundary of $f(\Omega)$ lies at the boundary of the rectangle $[0,I]\times [0,1]$. The set $f(\Omega)$ lies in $[0,I]\times [0,1]$, because $s_n(\Omega)$ lies in $[0,I^*_n]\times [0,1]$ for every $n\geq 0$, and $s_{k_n}$ converges to $f$ by \Lr{s conv}. There is a unique set satisfying the aforementioned properties of $f(\Omega)$, namely $(0,I)\times (0,1)$. Therefore, $f$ maps $\Omega$ onto $(0,I)\times (0,1)$. 

Since $\partial \Omega$ is a Jordan curve, $f$ extends to a homeomorphism between $\overline{\Omega}$ and $[0,I]\times [0,1]$ by Caratheodory's theorem, and maps $x_i$, $i=1,2,3,4$ to $y_i$ by \Lr{corners}. As mentioned in \Sr{complex}, $I$ equals the reciprocal of the extremal length between $T$ and $B$, and the above properties uniquely determine $f$. Consequently, all subsequential limits of $s_n$ coincide with $f$. Hence $s_n$ converges in $C^{\infty}(\Omega)$ to $f$, as desired.
\end{proof}

Having proved the convergence of $s_n$ to $f$ we can now strengthen the statement of \Lr{max side}, namely we can prove that the maximum side length of a square $S_e$ over all squares of $S$ converges to $0$. Indeed, for those edges lying in some compact subset $K$ of $\Omega$, the convergence to $0$ follows from \Lr{max side}. For those edges lying in $\Omega\setminus K$ we can use the fact that for every edge $e$ of our graph, $w_n(e)^2$ is equal to the area of the square $S_e$, to obtain that the sum $\sum_{e \subset \Omega\setminus K} w_n(e)^2$ coincides with the area of the union of the corresponding squares. Choosing $K$ appropriately we can ensure that $f$ maps $\Omega\setminus K$ so close to the boundary that the area of $f(\Omega\setminus K)$ is as small as we want. Arguing as in the proof of \Lr{local}, and using the uniform convergence of $s_n$ to $f$ on $K$ and the convergence of $I^*_n$ to $I$, we can deduce that the area of $s_n(\Omega\setminus K)$ is as small as we want for every $n$ large enough. This easily proves the desired convergence of the maximum side length to $0$. The details are left to the reader. \\

Since the only limit point of $(I^*_n)$ is $I$, we obtain that $I^*_n$ converges to $I$. Moreover, $I$ is the reciprocal of the extremal length between $T$ and $B$ by the discussion in \Sr{complex}, and $R^{eff}_n=1/I^*_n$. As a corollary we obtain

\begin{corollary}
The effective resistance $R^{eff}_n$ between $T_n$ and $B_n$ converges to the extremal length between $T$ and $B$.
\end{corollary}

\section*{Acknowledgements}
We would like to thank Gr{\'e}gory Miermont for introducing the problem, and Nicolas Curien for a helpful discussion.

\bibliographystyle{plain}
\bibliography{collective}

\begin{thebibliography}{10}

\bibitem{AhlConformal}
Lars~Valerian Ahlfors.
\newblock {\em Conformal invariants: topics in geometric function theory}.
\newblock McGraw-Hill, 1973.

\bibitem{BeSchrHar}
I.~Benjamini and O.~Schramm.
\newblock Random walks and harmonic functions on infinite planar graphs using
  square tilings.
\newblock {\em The Annals of Probability}, 24(3):1219--1238, 1996.

\bibitem{Bill}
Patrick Billingsley.
\newblock {\em Convergence of probability measures}.
\newblock Wiley, 1999.

\bibitem{Brandt}
A.~Brandt.
\newblock {Estimates for difference quotients of solutions of Poisson type
  difference equations}.
\newblock {\em Mathematics of Computation}, 20(96):473--499, 1966.

\bibitem{SqTil}
R.~L. Brooks, C.~A.~B. Smith, A.~H. Stone, and W.~T. Tutte.
\newblock The dissection of rectangles into squares.
\newblock {\em Duke Mathematical Journal}, 7(1):312--340, 1940.

\bibitem{RBMWeak}
K.~Burdzy and Z.~Chen.
\newblock Weak convergence of reflected brownian motions.
\newblock {\em Electronic Communications in Probability}, 3(4):29--33, 1998.

\bibitem{CaFlPaSqu}
J.~W. Cannon, W.~J. Floyd, and W.~R. Parry.
\newblock {Squaring rectangles: The finite Riemann mapping theorem.}
\newblock {Abikoff, William (ed.), The mathematical legacy of Wilhelm Magnus.
  Contemp.\ Math.\ 169}, 1994.

\bibitem{ChelSmi}
D.~Chelkak and S.~Smirnov.
\newblock Discrete complex analysis on isoradial graphs.
\newblock {\em Advances in Mathematics}, 228(3):1590--1630, 2011.

\bibitem{ColSteCir}
C.~R. Collins and K.~Stephenson.
\newblock {A circle packing algorithm}.
\newblock {\em Computational Geometry}, 25(3):233--256, 2003.

\bibitem{Courant}
R.~Courant, K.~Friedrichs, and H.~Lewy.
\newblock {{\"U}ber die partiellen Differenzengleichungen der mathematischen
  Physik}.
\newblock {\em Mathematische annalen}, 100(1):32--74, 1928.

\bibitem{DHeRodin}
P.~Doyle, Z.~X. He, and B.~Rodin.
\newblock Second derivatives of circle packings and conformal mappings.
\newblock {\em Discrete \& Computational Geometry}, 11(1):35--49, 1994.

\bibitem{DuffinEx}
R.~J. Duffin.
\newblock The extremal length of a network.
\newblock {\em Journal of Mathematical Analysis and Applications},
  5(2):200--215, 1962.

\bibitem{planarPB}
A.~Georgakopoulos.
\newblock {The boundary of a square tiling of a graph coincides with the
  Poisson boundary}.
\newblock {\em Inventiones mathematicae}, 203(3):773--821, 2016.

\bibitem{HeRodin}
Z.~X. He and B.~Rodin.
\newblock {Convergence of circle packings of finite valence to Riemann
  mappings}.
\newblock {\em Communications in Analysis and Geometry}, 1(1):31--41, 1993.

\bibitem{HeSchRie}
Z.~X. He and O.~Schramm.
\newblock {On the convergence of circle packings to the Riemann map}.
\newblock {\em Inventiones mathematicae}, 125(2):285--305, 1996.

\bibitem{HeSchInf}
Z.~X. He and O.~Schramm.
\newblock {The $C^\infty$-convergence of hexagonal disk packings to the Riemann
  map}.
\newblock {\em Acta Mathematica}, 180(2):219--245, 1998.

\bibitem{KaraSh}
Ioannis Karatzas and Steven Shreve.
\newblock {\em Brownian Motion and Stochastic Calculus}.
\newblock Springer, 1991.

\bibitem{Krantz}
Steven~G. Krantz.
\newblock {\em Geometric function theory: explorations in complex analysis}.
\newblock 2006.

\bibitem{LawEst}
G.~F. Lawler.
\newblock {Estimates for differences and Harnack inequality for difference
  operators coming from random walks with symmetric, spatially inhomogeneous,
  increments}.
\newblock {\em Proceedings of the London Mathematical Society}, 3(3):552--568,
  1991.

\bibitem{LF}
Jacqueline Lelong-Ferrand.
\newblock {\em {Repr{\'e}sentation conforme et transformations {\`a}
  int{\'e}grale de Dirichlet born{\'e}e}}, volume~22.
\newblock Gauthier-Villars, 1955.

\bibitem{MorPer}
Peter M{\"o}rters and Yuval Peres.
\newblock {\em Brownian motion}.
\newblock Cambridge University Press, 2010.

\bibitem{RBMConf}
M.~Pascu.
\newblock {Scaling coupling of reflecting Brownian motions and the hot spots
  problem}.
\newblock {\em Transactions of the American Mathematical Society},
  354(11):4681--4702, 2002.

\bibitem{Pom}
Christian Pommerenke.
\newblock {\em {Boundary Behaviour of Conformal Maps}}.
\newblock Springer-Verlag, 1992.

\bibitem{RodSu}
B.~Rodin and D.~Sullivan.
\newblock {The convergence of circle packings to the Riemann mapping}.
\newblock {\em Journal of Differential Geometry}, 26(2):349--360, 1987.

\bibitem{BabyRudin}
Walter Rudin.
\newblock {\em {Principles of mathematical analysis}}.
\newblock McGraw-Hill, 1964.

\bibitem{PapaRudin}
Walter Rudin.
\newblock {\em {Real and Complex analysis}}.
\newblock McGraw-Hill, 1987.

\bibitem{Steph}
K.~Stephenson.
\newblock {A probabilistic proof of Thurston's conjecture on circle packings}.
\newblock {\em Rendiconti del Seminario Matematico e Fisico di Milano},
  66(1):201--291, 1996.

\bibitem{ThuFin}
W.~Thurston.
\newblock {The finite Riemann mapping theorem}.
\newblock In {\em Invited talk, An International Symposium at Purdue University
  on the occasion of the proof of the Bieberbach conjecture}, 1987.

\end{thebibliography}

\end{document}